\documentclass[11pt,reqno]{amsart}

\usepackage{amsmath,amsthm,amssymb,enumerate,tocvsec2,color}
\usepackage[tmargin=1.4in,bmargin=1.3in,rmargin=1.3in,lmargin=1.3in]{geometry}
\usepackage[breaklinks=true]{hyperref}

\theoremstyle{plain}
\newtheorem{theorem}{Theorem}[section]

\newtheorem{corollary}[theorem]{Corollary}
\newtheorem{proposition}[theorem]{Proposition}
\newtheorem{lemma}[theorem]{Lemma}

\theoremstyle{definition}
\newtheorem{definition}[theorem]{Definition}
\newtheorem{remark}[theorem]{Remark}
\newtheorem{question}[theorem]{Question}

\numberwithin{equation}{section}

\newcommand{\bc}{\mathbf{c}}
\newcommand{\bd}{\mathbf{d}}
\newcommand{\bdelta}{\boldsymbol{\delta}}
\newcommand{\be}{\mathbf{e}}
\newcommand{\bm}{\mathbf{m}}
\newcommand{\bn}{\mathbf{n}}
\newcommand{\bone}{\mathbf{1}}

\newcommand{\bt}{\mathbf{t}}
\newcommand{\bu}{\mathbf{u}}
\newcommand{\bv}{\mathbf{v}}
\newcommand{\bw}{\mathbf{w}}
\newcommand{\bx}{\mathbf{x}}

\newcommand{\bzero}{\mathbf{0}}
\newcommand{\C}{\mathbb{C}}
\newcommand{\cC}{\mathcal{C}}
\newcommand{\cP}{\mathcal{P}}
\newcommand{\cU}{\mathcal{U}}
\newcommand{\diag}{\mathop{\mathrm{diag}}\nolimits}

\newcommand{\Id}{\mathrm{Id}}
\newcommand{\one}[1]{\bone_{#1 \times #1}}
\newcommand{\R}{\mathbb{R}}
\newcommand{\rd}{\mathrm{d}}
\newcommand{\rk}{\mathop\mathrm{rank}}
\newcommand{\std}{\,\rd}
\newcommand{\stratumsymb}{\mathcal{S}}
\newcommand{\stratum}[1]{\stratumsymb_{#1}}
\newcommand{\Z}{\mathbb{Z}}

\newcommand{\up}[1][\pi]{\Sigma^\uparrow_{#1}}
\newcommand{\down}[1][\pi]{\Sigma^\downarrow_{#1}}
\newcommand{\ccup}[1][\pi]{\Theta^\uparrow_{#1}}
\newcommand{\ccdown}[1][\pi]{\Theta^\downarrow_{#1}}

\renewcommand{\subset}{\subseteq}

\renewcommand{\leq}{\leqslant}
\renewcommand{\geq}{\geqslant}

\newcommand{\editAB}[1]{{\color{black}#1}} 

\begin{document}
\title[Matrix positivity preservers in fixed dimension. II]%
{Matrix positivity preservers in fixed dimension. II: %
positive definiteness and\\
strict monotonicity of Schur function ratios}

\author{Alexander Belton}
\address[A.~Belton]{School of Engineering, Computing and Mathematics,
University of Plymouth, Plymouth, UK}
\email{\tt alexander.belton@plymouth.ac.uk}

\author{Dominique Guillot}
\address[D.~Guillot]{University of Delaware, Newark, DE, USA}
\email{\tt dguillot@udel.edu}

\author{Apoorva Khare}
\address[A.~Khare]{Indian Institute of Science;
Analysis and Probability Research Group; Bangalore, India}
\email{\tt khare@iisc.ac.in}

\author{Mihai Putinar}
\address[M.~Putinar]{University of California at Santa Barbara, CA,
USA and Newcastle University, Newcastle upon Tyne, UK} 
\email{\tt mputinar@math.ucsb.edu, mihai.putinar@ncl.ac.uk}

\date{\today}

\begin{abstract}
We continue the study of real polynomials acting entrywise on
matrices of \editAB{fixed dimension to preserve positive semidefiniteness,
together with the related} analysis of order properties of Schur polynomials.

\editAB{Previous work has shown that, given a real polynomial
with positive coefficients that is perturbed by adding a
higher-degree monomial, there exists a negative lower bound 
for the coefficient of the perturbation which characterises when the
perturbed polynomial remains positivity preserving.

We show here that, if the perturbation coefficient is strictly
greater than this bound then the transformed matrix becomes
positive definite given a simple genericity condition that can be readily verified.
We identity a slightly stronger \editAB{genericity} condition that
ensures positive definiteness occurs at the boundary.}

The analysis is \editAB{complemented} by computing the rank of
\editAB{the transformed matrix in terms of the location of the original
matrix in a} Schubert cell-type stratification \editAB{that we have
introduced and explored previously}.
The proofs require enhancing to strictness a Schur monotonicity
result of Khare and Tao, to show that the ratio of
\editAB{Schur polynomials is} strictly increasing along each coordinate
on the positive \editAB{orthant} and non-decreasing on its closure
\editAB{whenever the defining tuples satisfy a coordinate-wise
domination condition}.
\end{abstract}


\maketitle


\section{Background and setup}

The study of entrywise positivity preservers involves understanding the
structure of functions \editAB{of the form} $f : I \to \R$,
\editAB{for some complex} domain $I$, such that, if a complex Hermitian
matrix $A = ( a_{i j} )$ with entries in $I$ is positive
semidefinite then so is \editAB{the matrix}
$f[ A ] := \bigl( f( a_{i j} ) \bigr)$; \editAB{when $f$ is a power
function, so that $f( x ) \equiv x^\alpha$ for some $\alpha$,
we also use the Schur product notation $f[ A ] = A^{\circ \alpha}$.}

This subject has a rich history, beginning with the Schur
product theorem~\cite{Schur}, which implies that all
\editAB{functions represented by} power
series with non-negative coefficients preserve positivity
\editAB{in this sense for square matrices of arbitrary size}. 
The converse, that there are no other preservers in all dimensions,
was first shown by Schoenberg~\cite{Schoenberg42} for
continuous functions defined on $I = [-1,1]$, and subsequently by several
others. \editAB{For domains of the form}
$I = (-\rho,\rho)$, with $0 < \rho \leq \infty$, we
mention Rudin~\cite{Rudin59} and recent work~\cite{BGKP-hankel} for
variants with greatly reduced test sets in each dimension. \editAB{The}
book \cite{K} \editAB{contains} additional details and references.

The situation is more involved in \editAB{a} fixed dimension $N$,
where \editAB{the} complete classification of the entrywise positivity preservers
remains open to date even for \editAB{$3 \times 3$ matrices, that is, when}
$N = 3$. For matrices with positive entries, the real powers
which are entrywise \editAB{positivity} preservers were classified by
FitzGerald and Horn in~\cite{FitzHorn}:
\editAB{these are the non-negative integers and
all real powers beyond the threshold $N - 2$, that is, elements of}
the set $\Z_+ \cup [ N - 2, \infty )$.
If one considers polynomial preservers instead then
\editAB{no such preservers were known in fixed dimension
beyond the case of non-negative coefficients, until the previous part of}
this work~\cite{BGKP-fixeddim}, \editAB{which was} subsequently
extended by Khare \editAB{and} Tao~\cite{KT}.

\subsection{Polynomial preservers yield positive definite matrices}

We lay out \editAB{here} in a condensed form the \editAB{key results}
from \cite{BGKP-fixeddim,KT} \editAB{that are relevant to our work here},
and  \editAB{provide from this context} the first novel observation of this paper.

By a \editAB{result} of Loewner (see \cite{horn}), if $\rho > 0$
\editAB{and the smooth function $f : ( 0, \rho ) \to \R$ is such that
$f[ A ]$ is positive semidefinite for any positive semidefinite}
$A \in ( 0, \rho )^{N \times N}$, then $f$, $f$', \ldots, $f^{( N - 1 )}$ are
non-negative on $( 0, \rho )$, but \editAB{this} need not hold for any higher
derivative of~$f$. \editAB{More generally}, if~$f$ is a real polynomial
preserver with exactly $N + 1$ monomial terms, then 
\editAB{the first}~$N$ non-zero Maclaurin coefficients of $f$ are
positive. The question of \editAB{whether the leading coefficient could be
negative} was eventually answered \editAB{positively}
in~\cite{BGKP-fixeddim,KT} 
\editAB{with an explicit sharp negative lower bound}
in several slightly different settings. \editAB{We begin here}
fixing some notation, \editAB{introducing these settings and then
providing the common bound which holds for} all of them.

\begin{definition}\label{Dsetup}
Given a domain $I \subset \C$  and \editAB{positive integers $N$ and $k$ with}
$k \leq N$, denote by $\cP_N^k( I )$ the set of positive semidefinite
$N \times N$ matrices with entries in $I$ and rank at most $k$;
\editAB{recall that any positive semidefinite complex matrix is
automatically Hermitian. For convenience, we} also set
$\cP_N( I ) := \cP_N^N ( I )$.
The Loewner \editAB{partial} order on $N \times N$ Hermitian
matrices is \editAB{defined by setting} $A \geq B$ if \editAB{and only if}
$A - B \in \cP_N( \C )$.

\editAB{For any $\rho > 0$,} let $\overline{D}( 0, \rho )$ denote the closed disc
in~$\C$ with center~$0$ and radius $\rho$. \editAB{We are interested in the
entrywise action of the function
\begin{equation}\label{Ef}
f ( z ) := \sum_{j = 0}^{N - 1} c_j z^{n_j}  + c' z^M = h( z ) + c' z^M
\end{equation}
on some suitable set of test} matrices
$\cP_0 \subset \cP_N\bigl( \overline{D}\bigl( ( 0, \rho ) \bigr)$,
\editAB{where the number of terms $N$ is a positive integer,
the coefficients $c_0$, \ldots, $c_{N - 1}$ and $c'$ are real numbers and
the powers are arranged in increasing order:}
$n_0 < n_1 < \cdots < n_{N - 1} < M$.
\editAB{The test set $\cP_0$ is may depend on the form of $f$, as follows.}
\begin{enumerate}
\item \editAB{The minimal subset} $\cP_0 = \cP_N^1\bigl( ( 0, \rho ) \bigr)$,
for arbitrary real powers \editAB{$n_0$, \ldots, $n_{N - 1}$ and~$M$.}
\item \editAB{A subset $\cP_0$ such that}
$\cP_N^1\bigl( ( 0, \rho ) \bigr) \subset \cP_0 \subset \cP_N^1\bigl( [ 0, \rho ] \bigr)$,
for non-negative powers \editAB{$n_0$, \ldots, $n_{N - 1}$ and $M$}. Here
\editAB{and elsewhere} we set $0^0 := 1$.
\item \editAB{A subset $\cP_0$ such that}
$\cP_N^1\bigl( (0,\rho) \bigr) \subset \cP_0 \subset \cP_N( [ 0, \rho ] \bigr)$,
where \editAB{$n_0$, \ldots, $n_{N - 1}$ and $M$ are elements of the set}
$\Z_+ \cup [ N - 1, \infty )$. \footnote{\editAB{As mentioned previously},
it is known~\cite{FitzHorn}
that all real powers $\alpha \geq N - 2$ preserve positivity \editAB{when
acting entrywise} on $\cP_N\bigl( [ 0, \rho ] \bigr)$,
but we need more for our purposes, namely,
powers \editAB{that preserve the Loewner order} on $\cP_0$:
if $A$, $B \in \cP_0$ with $A - B \in \cP_N\bigl( [ 0, \rho ] \bigr)$
then $A^{\circ \alpha} - B^{\circ \alpha} \in \cP_N( \R )$.
\editAB{See \cite[Theorem~5.1(ii)]{Hiai}.}}

\item \editAB{A subset $\cP_0$ such that}
$\cP_N^1\bigl( ( 0, \rho ) \bigr) \subset \cP_0 \subset %
\cP_N\bigl( \overline{D}( 0, \rho ) \bigr)$, where
\editAB{$n_0$, \ldots, $n_{N - 1}$} are
successive non-negative integers
\editAB{(so that $n_j = n_0 + j$ for $j = 0$, \ldots, $N - 1$)}
and $M$ is an integer.
\end{enumerate}
\end{definition}

\editAB{In the complex case (4) above, if the 
polynomial $f$ has the form (\ref{Ef}) with coefficients
$c_0$, \ldots, $c_{N - 1} > 0$
and $c_M < 0$, and the powers $n_0$, \ldots, $n_{N - 1}$ are not
successive non-negative integers then $f$ does not preserve positive
semidefiniteness entrywise on $\cP_0$ for some $M > n_{N - 1}$: see~\cite[Proposition~7.1]{KT}.}

Having \editAB{described these possibilities}, we recall the 
\editAB{corresponding} classification of
entrywise polynomial preservers.

\begin{theorem}[{\cite[Theorem~1.1]{BGKP-fixeddim} and
\cite[Section~1.3]{KT}}]\label{Tthreshold}
\editAB{Let $f$ be as in (\ref{Ef}), let $\rho > 0$ and set
\begin{equation}\label{DC}
\cC = \cC( f, \rho ) := %
\sum_{j = 0}^{N - 1} \frac{V( \bn_j )^2}{V( \bn )^2} %
\frac{\rho^{M - n_j}}{c_j},
\end{equation}
where the Vandermonde determinant
\[
V( \bm ) := \prod_{1 \leq k < l \leq N} ( m_l - m_k ) \qquad %
\textrm{for any } \bm = ( m_1, \ldots, m_N )
\]
and the $N$-tuples
\begin{equation}\label{Erankone}
\bn_j := ( n_0, \ldots, \widehat{n_j}, \ldots, n_{N - 1}, M ) %
\qquad \textrm{and} \qquad \bn := ( n_0, \ldots, n_{N - 1} ),
\end{equation}
where $\widehat{n_j}$ indicates that $n_j$ is omitted.
Given a test set $\cP_0$ according to Definition~\ref{Dsetup},
the following are equivalent.}
\begin{enumerate}
\item The map $f[ - ]$ preserves positivity on $\cP_0$.
\item The coefficients of $f$ satisfy either
(a)~$c_0$, \ldots, $c_{N - 1}$, $c' \geq 0$, or
(b)~$c_0$, \ldots, $c_{N - 1} > 0$ and
$c' \geq -\cC^{-1}$.
\item The map $f[ - ]$ preserves positivity on the \editAB{subset of}
Hankel matrices in $\cP_N^1\bigl( ( 0, \rho ) \bigr)$.
\end{enumerate}
\end{theorem}

\editAB{Fundamentally, our work involves} the constructive analysis of the
largest \editAB{eigenvalue for linear pencils} of Hermitian matrices
\editAB{of the form}
\[
h[ A ] - \lambda A^{\circ M},
\]
where $h$ is \editAB{the unperturbed} polynomial adapted to the \editAB{size}
of \editAB{the positive matrix} $A$ and the power~$M$ exceeds the degree of~$h$.
One of the results we show in the present work is an
\editAB{enhancement of previous work to show the} positive definiteness of
$f[ A ]$ for generic $A$:

\begin{theorem}\label{T1}
\editAB{Let $f$, $\cC$ and $\cP_0$ be} as in Theorem~\ref{Tthreshold},
\editAB{with $c_0$, \ldots, $c_{N - 1} > 0$ and $c' > \cC^{-1}$}. If all
\editAB{of the} rows of \editAB{$A \in \cP_0$} are distinct and
\editAB{$n_0 = 0$ when $A$ has a zero row} then $f[ A ]$ is positive definite.
\end{theorem}

This is stated and proved in Theorems~\ref{Tlmi} and~\ref{Tlmi2} below.

\editAB{To establish these two theorems, we rely on a lower-bound result, that if
a positive semidefinite matrix $A$ has distinct rows then it has a rank-one lower
bound $\bu$, such that~$A \geq \bu \bu^T$, and $\bu$ may be chosen to have
distinct entries. In the complex setting this is elementary, but if $A$
has non-negative entries and $\bu$ is required to as well then we
establish the existence of such a lower bound using Perron--Frobenius theory.
This result, Theorem~\ref{Tdom}, may be of independent interest.}

\subsection{Strict monotonicity of Schur polynomial ratios}

Next we switch tracks and focus on Schur polynomials from an
\editAB{order} perspective. \editAB{While this may seem a non
\editAB{sequitur}, it is not: the proofs of Theorem~\ref{Tthreshold} in
\cite{KT} rely} crucially on
\begin{itemize}
\item[(i)] a combinatorial determinant formula involving Schur polynomials
(Theorem~\ref{Tjacobi-trudi}) and
\item[(ii)] a Schur monotonicity \editAB{lemma (Section~\ref{SSstrict}).}
\end{itemize}
\editAB{We will now introduce some notation to facilitate the statement
of the monotonicity lemma.}

\editAB{For any set of real numbers $S$, the collection of $N$-tuples
of distinct elements of~$S$ is denoted by $S^N_{\neq}$
and its subset of $N$-tuples with entries in increasing order
is denoted by~$S^N_<$.
Given vectors $\bu = ( u_i )_{i = 1}^N \in ( 0, \infty )^N$ and
$\bm = ( m_j )_{j = 1}^N \in \R^N$, we let the matrix
$\bu^{\circ \bm} := ( u_i^{m_j} )_{i, j = 1}^N$.} 

\begin{theorem}[Schur monotonicity lemma,
{\editAB{\cite[Corollary~8.7 and Proposition~8.1]{KT}}}]\label{TSchurmon}
\editAB{Let $\bm$, $\bn \in \R^N_<$ be such that $m_j \leq n_j$ for all $j$,
where $N \geq 1$.}
The symmetric function 
\[ 
f : ( 0, \infty )^N_{\neq} \to \R; \ %
\bu \mapsto \frac{\det \bu^{\circ \bn}}{\det \bu^{\circ \bm}}
\]
is non-decreasing in each coordinate. If, moreover,
\editAB{the entries of the vectors $\bm$ and $\bn$ are}
non-negative integers then $f$ \editAB{extends uniquely
to the whole of $( 0, \infty )^N$
and coordinate-wise monotonicity holds \editAB{everywhere}}.
\end{theorem}

\editAB{To see the connection with Schur, we note that when
$\bm$ and $\bn$ are composed of non-negative integers then
$f( \bu ) \equiv s_\bn( \bu ) / s_\bm( \bu )$, the ratio of
Schur polynomials $s_\bm$ and $s_\bn$ as defined in
(\ref{Dschur}) below.}

Theorem~\ref{TSchurmon} is interesting for multiple reasons. First, it
provided the missing ingredient required to extend the positivity
preserver results in~\cite{BGKP-fixeddim} to general polynomials
in~\cite{KT}.
Second, it led to novel characterizations in the theory of real
inequalities~\cite{KT}: of weak majorization, as well as of
majorization for all real tuples, extending the integer-tuple
case in~\cite{CGS,Sra}.
\editAB{Third,} this result admits several different proofs:
via a log-supermodularity phenomenon and totally positive matrices~\cite{KT},
using a result of Lam, Postnikov and Pylyavskyy~\cite{LPP} from
representation theory and the theory of symmetric functions~\cite{KT},
and relying on the theory of Chebyshev blossoming in M\"untz spaces,
as developed by Ait-Haddou
\editAB{and co-authors}~\cite{Blossom2, Blossom1}.

\editAB{In fact, the hypotheses of this theorem serve
to deliver a stronger conclusion and this is our second} main result:

\begin{theorem}\label{T2}
\editAB{With the hypotheses of} Theorem~\ref{TSchurmon},
\editAB{when $\bm$ and $\bn$ are distinct the function} $f$ is
\editAB{actually} strictly increasing in each coordinate.
\editAB{Moreover, when $\bm$ and $\bn$ also have non-negative-integer
entries, this} coordinate-wise strict monotonicity
holds \editAB{for the extension of~$f$ to} all of~$( 0, \infty )^N$.
\end{theorem}

In fact, we show a stronger result than the final assertion here, by
extending \editAB{the function}
to parts of the boundary of the positive orthant. Moreover,
it is not the generalized Vandermonde ratio with non-integer powers but
the Schur polynomial ratio with integer exponents whose strict
monotonicity has the more involved proof. See Theorems~\ref{Tkt}
and~\ref{Tkt2} \editAB{for details}.

Apart from its intrinsic interest, Theorem~\ref{T2} is the key to proving
\editAB{Theorem~\ref{T1}} and
\editAB{its variations in Section~\ref{Sstrict}}.
The proofs of both \editAB{main} results combine
\editAB{techniques from analysis}
with properties of Schur polynomials, which are inherently
algebraic objects with a representation-theoretic flavour. Our exploration
reinforces the need \editAB{for further study of} Schur functions from an
\editAB{analytical}
viewpoint. Prior work has already revealed the essential role of Schur functions in \editAB{the investigation of} positivity transforms \editAB{(see}
\cite{BGKP-fixeddim, KT} \editAB{and also~\cite{McSN}),}
and we can add two more
contributions from recent work \cite{horndet}. The first creates a bridge
between analysis and algebra: \editAB{the Schur polynomials lie within}
 the Maclaurin expansion of $\det f[ \bu \bv^T ]$
\editAB{for every} smooth function~$f$. The second walks across this bridge to contribute to algebra: the well-known determinant formula of Cauchy in
symmetric function theory, its extension by Frobenius, and a
determinant computation by Loewner \cite{horn} all admit a
common extension, to \editAB{power series over an arbitrary}
commutative ring.

While the main theme of our work is the classification of positivity transforms,
at least two ingredients in the proofs below may be of independent interest:
the strict monotonicity of certain \editAB{ratios} of Schur functions
\editAB{and} the continuity of \editAB{certain Rayleigh quotients}
on isogenic strata of positive matrices.

\editAB{One conclusion that may be drawn from} the present article is that
applications of Schur functions \editAB{to topics beyond algebra} are far from being fully explored. Further discoveries
\editAB{and more surprises undoubtedly lie in wait.}

\subsection*{Organisation of the remainder of this paper}

\editAB{Section~\ref{Sstrict} contains the statements and proofs of
extended versions of the two new theorems stated above,
Theorems~\ref{T1} and~\ref{TSchurmon}. This section concludes
by resolving the question of whether Loewner's necessary condition
for smooth functions to preserve positive semidefiniteness in fixed
dimension is also sufficient.}

\editAB{In Section~\ref{Srank}, we 
recall the isogenic block stratification from~\cite{BGKP-pmp, BGKP-strata}
and use this to find the rank} of the matrix $f [A ]$ for $A$ in
\editAB{any given} stratum and $f$ as in
\editAB{Theorem~\ref{Tthreshold}(2)(b)}.

\editAB{We conclude with} Section~\ref{Scont}, in which we recall
\editAB{the interpretation}
from~\cite{BGKP-fixeddim} of the
\editAB{bound~$\cC$ in terms of a Rayleigh quotient.}
We prove that this Rayleigh quotient is
continuous as a function of \editAB{the underlying matrix}
when restricted to each isogenic stratum.

For the reader's convenience, we append before the bibliography a list of
symbols used throughout this article.

\section{Strictness of linear matrix inequalities for Hadamard
powers, and the Schur strict monotonicity lemma}\label{Sstrict}

In this section, we \editAB{obtain two variations on
Theorem~\ref{Tthreshold}. We note first the following consequence
of this theorem.}

\begin{corollary}\label{Clmi}
\editAB{Let $f$, $\cC$ and $\cP_0$ be} as in Theorem~\ref{Tthreshold}.
If $c_0$, \ldots, $c_{N - 1} > 0$ then
\begin{equation}\label{Elmi}
A^{\circ M} \leq \cC \sum_{j=0}^{N-1} c_j A^{\circ n_j} 
\qquad \textrm{for any } A \in \cP_0,
\end{equation}
where $\leq$ denotes the Loewner ordering,
\editAB{and the} constant $\cC$ is sharp.
\end{corollary}

It follows immediately from this Corollary that \editAB{the matrix
\[
f[ A ] = \sum_{j = 0}^{N - 1} c_j A^{\circ n_j} + c' A^{\circ M}
\]
is} positive semidefinite, \editAB{whenever
$c_0$, \ldots, $c_{N-1} > 0$ and $c' \geq -\cC^{-1}$, for any
$A \in \cP_0$. We introduce and recall some notation for two important
boundary cases:
\[
g( z ) = \sum_{j = 0}^{N - 1} c_j z^{n_j} -\cC^{-1} z^M %
\qquad \textrm{and} \qquad %
h( z ) = \sum_{j = 0}^{N - 1} c_j z^{n_j}.
\]}
It is natural to ask when \editAB{the matrices $f[ A ]$, $g[ A ]$ and $h[ A ]$}
are positive definite.
The following strengthening of Theorem~\ref{Tthreshold} shows that these
matrices are generically positive definite in a strong sense, and zero
only in the one-dimensional, degenerate case.

\begin{theorem}\label{Tlmi}
\editAB{Let $f$ and $\cP_0$ be as in Definition~\ref{Dsetup}(4), so
that $n_0$ and $M$ are non-negative integers and $n_j = n_0 + j$
for $j = 0$, \ldots, $N -1$. Suppose $c_0$, \ldots, $c_{N - 1} > 0$
and $c' > -\cC^{-1}$, where $\cC$ is as in (\ref{DC}).}
\begin{enumerate}
\item \editAB{Let $A \in \cP_0$ and suppose $n_0 = 0$
if $A$ has a zero row.} The following are equivalent.
\begin{enumerate}
\item There exists a vector $\bu \in \C^N$ with distinct entries such
that $A \geq \bu \bu^*$ and \editAB{$\bu$ has a zero entry if and only if
$A$ has a zero row.}

\item All \editAB{of the} rows of $A$ are \editAB{distinct}.

\item The matrix $h[ A ]$ is positive definite.

\item The inequality~(\ref{Elmi}) is strict, that is, $f[ A ]$ is
positive definite.
\end{enumerate}
\item Suppose $A \in \cP_0$ has a row with distinct entries and $n_0 = 0$ if any
entry in this row is zero. Then $g[ A ]$ is positive definite.
\end{enumerate}
\editAB{Furthermore,} equality in~(\ref{Elmi}) is attained on~$\cP_0$
\editAB{if and only if either $N = 1$ and $A = \rho$, or
$n_0 > 0$ and $A = \bzero_{N \times N}$.}
\end{theorem}

Note that part~(1)(a) of Theorem~\ref{Tlmi} does not depend on the
coefficients $c_0$, \ldots, $c_{N - 1}$ and~$c'$,
\editAB{and that the existence of $\bu$ follows immediately
from Proposition~\ref{Ppos-constant} if $A$ is positive definite.}
Note also that ``row''
may be replaced with ``column'' \editAB{throughout, as all the
matrices are Hermitian.}

\editAB{Theorem~\ref{Tlmi2} below provides a variation on 
Theorem~\ref{Tlmi} for the other three settings of Definition~\ref{Dsetup}.}

The proof of Theorem~\ref{Tlmi} relies on the following preliminary
observations.

\begin{lemma}\label{Llmi}
Suppose $N \geq 1$ and $C$, $D \in \cP_N( \C )$ with $C \geq D$. Then
$C - t D$ has the same kernel and rank as~$C$ for all
$t \in [ 0, 1 )$.
\end{lemma}

\begin{proof}
Fix $t \in ( 0, 1 )$. If $C \bu = 0$ for some $\bu \in \C^N$, then, as
$0 \leq C - t D \leq C$, it follows that
\[
0 \leq \bu^* ( C - t D ) \bu \leq \bu^* C \bu = 0,
\]
so $\ker C \subset \ker ( C - t D )$. Conversely, if
$( C - t D ) \bu = 0$ for some $\bu \in \C^N$, then
\[
0 = \bu^* ( C - t D ) \bu = \bu^* C \bu - t ( \bu^* D \bu )
\quad \implies \quad
\bu^* C \bu = t ( \bu^* D \bu ).
\]
Now, if $\bu^* C \bu > 0$ \editAB{then $\bu^* D \bu > 0$, so}
\[
\bu^* \Bigl( C - \frac{1 + t}{2} D \Bigr) \bu = %
\frac{t - 1}{2} \bu^* D \bu < 0,
\]
which is impossible as $0 \leq C - D \leq C - \frac{1 + t}{2} D$. Thus
$C \bu = 0$, proving the reverse inclusion. We are now done, by the
rank-nullity theorem.
\end{proof}

\begin{proposition}[{\cite[Proposition~4.2]{BGKP-fixeddim}}]\label{Ppos-constant}
Suppose $N \geq 1$ and $C$, $D \in \cP_N( \C )$. The following are
equivalent.
\begin{enumerate}
\item If $\bv^* C \bv = 0$ for some $\bv \in \C^N$, then
$\bv^* D \bv = 0$.

\item \editAB{The inclusion} $\ker C \subset \ker D$ \editAB{holds.}

\item There exists a constant $t > 0$ such that $C \geq t D$.
\end{enumerate}
\end{proposition}

\editAB{It follows immediately from the previous result
that if $C$, $D \in \cP_N$ with $C \geq D$ and $D$ is positive definite,
so invertible, then $C$ is also invertible, so positive definite.}

While Theorem~\ref{Tlmi}(2) is a result on positive definiteness, its
proof uses connections to Schur polynomials and Young tableaux. The key
step in this respect is Theorem~\ref{Tjacobi-trudi}, which requires the following
definition (\editAB{which adopts a different convention to that often found in the} literature \cite{Macdonald}).

\begin{definition}
\editAB{As above, if $S$ is any subset of real numbers, we let
$S^N_<$ denote the set of all increasing $N$-tuples of the form
$\bn = ( n_0 < \ldots < n_{N - 1} )$ with entries in $S$. For such an
$N$-tuple $\bn$, we let $| \bn | := n_0 + \cdots + n_{N - 1}$.

Given any $\bn \in ( \Z_+ )^N_<$,} the corresponding
\emph{Schur polynomial} $s_\bn( u_1, \ldots, u_N )$
is the unique polynomial extension of \editAB{the rational expression}
\begin{equation}\label{Dschur}
s_\bn( u_1, \ldots, u_N ) := %
\frac{\det ( u_i^{n_{j - 1}} )_{i, j = 1}^N}%
{\det ( u_i^{j - 1} )_{i, j =1}^N}.
\end{equation}
Note that the denominator is precisely the Vandermonde determinant
\[
V( \bu ) = V( u_1, \ldots, u_N ) := \det ( u_i^{j - 1} )_{i, j = 1}^N %
= \prod_{1 \leq k < l \leq N} ( u_l - u_k )
\]
\editAB{and we can write $s_\bn( \bu ) V( \bu ) = \det \bu^{\circ \bn}$,
where}
the matrix $\bu^{\circ \bn} := ( u_i^{n_{j - 1}} )_{i, j = 1}^N$.
\editAB{Since the right-hand side of (\ref{Dschur}) is unchanged after
swapping any two elements of $\bu$, each Schur polynomial is a
symmetric function.

For any $q \neq 0$ we have the product identity
\cite[((7.105)]{Stan2}
\begin{equation}\label{Eproduct}
s_\bn( 1, q, \ldots, q^{N - 1} ) = %
\frac{\det\bigl( ( q^{n_{j - 1}} )^{i - 1} \bigr)_{i, j = 1}^N}%
{\det( q^{j - 1} )^{i - 1} )_{i, j = 1}^N} = %
\prod_{1 \leq k < l \leq N} %
\frac{q^{n_{l  - 1}} - q^{n_{k - 1}}}{q^{l -1} - q^{k - 1}}
\end{equation}
as the numerator and
denominator are both Vandermonde determinants. Taking $q \to 0$
leads to the specialisation
\[
s_\bn( u \bone_T^N ) = u^{| \bn - \bdelta |} %
\prod_{1 \leq k < l \leq N} \frac{n_{l - 1} - n_{k - 1}}{l - k} = %
u^{| \bn - \bdelta |} \frac{V( \bn )}{V( \bdelta )}
\qquad \textrm{for all } u,
\]
where $\bdelta := ( 0, 1, 2, \ldots, N - 1 )$.}
As is well known \cite[Chapter~I, Equation~(5.12)]{Macdonald},
\editAB{thanks to Littlewood we have the identity
\begin{equation}\label{Elittlewood}
s_\bn( \bu ) = \sum_\bt \bu^\bt,
\end{equation}
a sum of $s_\bn( \bone_N^T) = V( \bn ) / V( \bdelta )$ monomials,
where the monomial $\bu^\bt := \prod_{j = 1}^N u_j^{t_j}$
has degree $| \bn - \bdelta |$ and the sum is taken over all
semistandard Young tableau $\bt$ of shape $\bn - \bdelta$.

In particular, if $\bn_j$ and
$\bn = n_0 \bone_N^T + \bdelta$ are as in~(\ref{Erankone}) and
we let $\bn_j' := \bn_j - n_0 \bone_N^T$ then
\begin{equation}\label{Esprime}
s_{\bn_j'}( \sqrt{\rho} \, \bone_N^T )^2 = %
\rho^{M - n_0 - j} \frac{V( \bn_j )^2}{V( \bn )^2} \qquad %
\textrm{for any } \rho > 0.
\end{equation}
Furthermore, it may be shown by the hook-content formula
\cite[Theorem~7.21.2]{Stan2} that
\[
\frac{V( \bn_j )}{V( \bn )} = s_{\bn_j}( \bone_N^T ) = %
s_{\bn_j'}( \bone_N^T ) = {M \choose j}{M - j - 1 \choose N - j - 1}.
\]}
\end{definition}

\begin{theorem}[\cite{KT}]\label{Tjacobi-trudi}
\editAB{Let $S$ be a finite set of real numbers of cardinality at least $N$ and suppose
\[
F( x ) = \sum_{n \in S} c_n x^n,
\]
where each coefficient $c_n$ is real. If $\bu \in \C^N$ then
\begin{equation}\label{Ejacobi-trudi2}
\det F[ \bu \bu^* ] = \sum_{\bn \in S^N_<} | \det \bu^{\circ \bn} |^2
\prod_{n \in \bn} c_n.
\end{equation}
In particular, if the elements of $S$ are non-negative integers then
\begin{equation}\label{Ejacobi-trudi}
\det F[ \bu \bu^* ] = \sum_{\bn \in S^N_<} | s_\bn( \bu ) |^2
| V( \bu ) |^2 \prod_{n \in \bn} c_n.
\end{equation}}
\end{theorem}

\editAB{We state and prove a short lemma before we give the proof of
Theorem~\ref{Tlmi}.}

\begin{lemma}\label{Lpd}
Suppose $\bw \in \C^N$ \editAB{has no zero entries. If
$B \in \cP_N( \C )$ is positive definite} then so is the Schur product
$( \bw \bw^*) \circ B$.
\end{lemma}

\begin{proof}
For any vector $\bv \neq \bzero$, we \editAB{have} that
$\bv \circ \overline{\bw} \neq \bzero$ \editAB{and therefore}
\[
\bv^* \bigl( ( \bw \bw^* ) \circ B ) \bv = %
( \bv \circ \overline{\bw} )^* B (v \circ \overline{\bw} ) > 0.
\qedhere
\]
\end{proof}

\begin{proof}[Proof of Theorem~\ref{Tlmi}]

\editAB{For part (1), we first} show that~(a) implies~(c).
Suppose~$\bu \in \C^N$ has distinct entries and is such that
$A \geq \bu \bu^*$. \editAB{Then}
$h[ \bu \bu^* ]$ is the sum of $N$ rank-one matrices with
linearly independent column spaces, since the 
\editAB{determinant of the matrix $( u_k^{n_0 + l - 1} )_{k, l = 1}^N$
is the product of a Vandermonde determinant and
$\prod_{k = 1}^N u_k^{n_0}$; recall that we take $0^0 = 1$.}
Thus, $h[ \bu \bu^* ]$ is non-singular and
so positive definite. As \editAB{noted above, entrywise powers
of non-negative integers are Loewner monotone on $\cP_N$, so}
$h[ A ] \geq h[ \bu \bu^* ]$
\editAB{and} $h[ A ]$ is also positive definite,
\editAB{by the remark after Proposition~\ref{Ppos-constant}.}

Next, \editAB{we note that (c) implies (b) because} the contrapositive is immediate.
\editAB{We now suppose that~(b) holds and deduce (a). Let
$\bu_1^T$, \ldots, $\bu_N^T$} denote the rows
of $A$. As~$\C^N$ is not a finite union of proper subspaces, we can
choose a vector $\bv \in \C^N$ that is not
orthogonal to \editAB{any vector of the form}
$\bu_j - \bu_k$ \editAB{with} $j \neq k$
nor any vector $\bu_j$ \editAB{that is} non-zero.
We set $\bw := A \overline{\bv}$ and
note that $\bw$ has distinct entries by the choice of $\bv$;
moreover, $\bw$ has a zero entry if and only if the corresponding
row of $A$ is zero. 
By Proposition~\ref{Ppos-constant}, \editAB{we have that}
$A \geq t \bw \bw^*$ for some
scalar~$t > 0$, so (a) follows by setting
$\bu := \sqrt{t} \bw$.

\editAB{Finally, that (c) implies~(d) follows from the remark after
Proposition~\ref{Ppos-constant} with $C = f[ A ]$ and $D = h[ A ]$
when $c' \geq 0$, and from} Corollary~\ref{Clmi} and
Lemma~\ref{Llmi} with \editAB{$C = h[ A ]$ and $D = \cC^{-1} A^{\circ M}$
when $c' < 0$. Conversely, that (d) implies~(c) follows from the same remark
when $c' \leq 0$, while if $c' > 0$, the implication follows from
Lemma~\ref{Llmi} with $C = f[ A ]$ and $D = ( c' + \cC^{-1} ) A^{\circ M}$,
together with Corollary~\ref{Clmi}.} This concludes the proof of part~(1).

To prove part~(2), we first show the rank-one case: if $A = \bu \bu^*$
for some column vector $\bu \in \overline{D}( 0, \sqrt{\rho} )^N$ 
\editAB{and $A$ has a row with distinct entries then $\bu$ has
distinct entries} and $g[ A ]$ is positive definite.

Suppose for contradiction that $\det g[ \bu \bu^* ] = 0$, and note
that, by specialising~(\ref{Ejacobi-trudi}) to the given parameters
\editAB{and using the fact that
$s_{\bn}( \bu ) = \prod_{j = 1}^N u_j^{n_0}$,}
\[
\sum_{j = 0}^{N - 1} \frac{| s_{\bn_j}( \bu ) |^2}{c_j} = %
\cC \prod_{j = 1}^N | u_j |^{2 n_0} = %
\prod_{j = 1}^N | u_j |^{2 n_0} %
\sum_{j = 0}^{N - 1} %
\frac{s_{\bn'_j}( \sqrt{\rho} \, \bone_N^T )^2}{c_j},
\]
\editAB{by (\ref{Esprime}),} where the partition $\bn_j$ is as
in~(\ref{Erankone}) \editAB{and} $\bn'_j := \bn_j - n_0 \bone_N^T$.

\editAB{We note from} the definitions that
$s_{\bn_j}( \bu ) = s_{\bn'_j}( \bu ) \prod_{j=1}^N u_j^{n_0}$.
\editAB{It now follows from} the triangle inequality and the
\editAB{Littlewood identity (\ref{Elittlewood})} that
\[
| s_{\bn'_j}( \bu ) |^2 \leq %
s_{\bn'_j}( \sqrt{\rho} \, \bone_N^T )^2,
\]
since $\bu \in \overline{D}( 0, \sqrt{\rho} )^N$, \editAB{and therefore}
$| s_{\bn'_j}( \bu ) | = s_{\bn'_j}( \sqrt{\rho} \, \bone_N^T ) = %
\editAB{\rho^{( M - n_0 - j ) / 2} V( \bn_j ) / V( \bn )}$
for all $j$.
\editAB{Another} application of the triangle inequality
implies that all monomials $\bu^\bt$ in the sum for $s_{\bn'_j}( \bu )$
have modulus $\rho^{( M - n_0 - j ) / 2}$ and so are equal
\editAB{(since the identity
$| z_1 + \cdots + z_n | = | z_1 | + \cdots + | z_n |$ implies that
the non-zero complex numbers $z_1$, \ldots, $z_n$ have the same argument).
Furthermore, as each entry of $\bu$ appears in some monomial, none of the
entries is zero.}

\editAB{If $M > n_0 + N$ then
$u_1^{M - n_0 - N} u_2 \cdots u_{N - j} u_k$}
is a monomial that occurs in the
\editAB{Littlewood formula for} $s_{\bn'_j}( \bu )$ for $k = 1$, \ldots, $N$,
and it follows that $u_1$, \ldots, $u_N$ are all equal.
\editAB{The edge case $M = n_0 + N$ must be dealt with separately,
but in this case $s_{\bn_j'}( \bu )$ is the sum of all monomials made up of
$N - j$ distinct entries of $\bu$ and the same conclusion holds.}
This contradicts the assumption that the entries of $\bu$ are
distinct, showing that $g[ \bu \bu^* ]$ is indeed positive definite.

Now suppose $A$ has a row $\bv^*$ with distinct entries; in particular,
the diagonal entry~$v'$ in~$\bv$ is real and positive. Set
$\bu := \bv / \sqrt{v'}$ and note that $A - \bu \bu^*$
has a zero row and column. If
\[
p_t[ B; R, \bd ] := t ( d_0 \one{N} + d_1 B + \cdots + %
d_{n - 1} B^{\circ ( N - 1 )} ) - B^{\circ ( N + R )}
\]
for \editAB{any} $\bd = ( d_0, \ldots, d_{N - 1} )$, then
\cite[(3.16)]{BGKP-fixeddim} yields \editAB{the identity}
\begin{align*}
p_t[ A; M - N, \bc ] & = p_t[ \bu \bu^*; M - N, \bc ] \\
 & + \int_0^1 ( A - \bu \bu^* ) \circ %
M p_{t / M}[ \lambda A + ( 1 - \lambda ) \bu \bu^*; M - N, \bc' ] %
\std \lambda,
\end{align*}
where $\bc' := ( c _1, 2 c_2, \ldots, ( N - 1 ) c_{N - 1} )$ and both
terms on the right-hand side are positive semidefinite,
\editAB{by \cite[(3.7)]{BGKP-fixeddim}}. Thus,
if~$t = \cC$ and $g_0( z ) := z^{-n_0} g( z )$ then
\[
g_0[ A ] = t^{-1} p_t[ A; M - N - n_0, \bc ] \geq %
t^{-1} p_t[ \bu \bu^*; M - N - n_0, \bc ] = g_0[ \bu \bu^* ],
\]
which is positive definite by the previous rank-one case. Thus $g_0[A]$
is positive definite, which completes the proof of part~(2) if $n_0 = 0$.
Otherwise, $n_0 > 0$ and all \editAB{the entries of~$\bu$} are non-zero
by hypothesis. In this case, the following calculation implies
\editAB{that the conclusion of part~(2) holds}:
\[
g[ A ] = A^{\circ n_0} \circ g_0[ A ] \geq %
( \bu^{\circ n_0} ( \bu^{\circ n_0})^* ) \circ g_0[ A ]
\]
and the right-hand side is positive definite by applying
\editAB{Lemma~\ref{Lpd}} with $\bw = \bu^{\circ n_0}$ and
$B = g_0[ A ]$.

The final assertion is immediate \editAB{when} $N = 1$, \editAB{so}
we conclude by showing equality does not hold in
(\ref{Elmi}) \editAB{whenever} $N > 1$ and $A \neq \bzero_{N \times N}$.
\editAB{As $A$ is positive semidefinite, some entry $x$ on the diagonal of $A$ is positive. Suppose}
$\bu \in ( 0, \sqrt{\rho} )^N$ \editAB{has} distinct entries,
one of which is~$\sqrt{x}$. The matrix $g[ \bu \bu^* ]$ is positive definite by
part~(2), so $g( x ) > 0$. Now equality holds in~(\ref{Elmi}) if and only
if $g[ A ] = 0$, but this working shows that at least one entry on the
main diagonal of $g[ A ]$ is strictly positive.
\end{proof}

Analogously to Theorem~\ref{Tlmi}, one has the following result for the
other test sets $\cP_\rho$ above.

\begin{theorem}\label{Tlmi2}
\editAB{Let $f$ and $\cP_0$ be as in Definition~\ref{Dsetup}(1--3)
and suppose $c_0$, \ldots, $c_{N - 1} > 0$ and $c' > -\cC^{-1}$,
where $\cC$ is as in (\ref{DC}).}
\begin{enumerate}
\item \editAB{Let $A \in \cP_0$ and suppose $n_0 = 0$
if $A$ has a zero row. The following are equivalent.}
\begin{enumerate}
\item There exists a vector $\bu \in  [ 0, \sqrt{\rho} ]^N$
with distinct entries such that $A \geq \bu \bu^*$ and
\editAB{$\bu$ has a zero entry if and only if $A$ has a zero row.}

\item All \editAB{of the} rows of $A$ are \editAB{distinct}.

\item The matrix $h[ A ]$ is positive definite.

\item The inequality~(\ref{Elmi}) is strict, that is, $f[ A ]$ is
positive definite.
\end{enumerate}
Moreover, (c) is equivalent to (d).
\item Suppose $A \in \cP_0$ has a row with distinct entries and $n_0 = 0$ if any
entry in this row is zero. Then $g[ A ]$ is positive definite.
\end{enumerate}
\editAB{Furthermore,} equality in~(\ref{Elmi}) is attained on~$\cP_0$
\editAB{if and only if either $N = 1$ and $A = \rho$, or
$n_0 > 0$ and $A = \bzero_{N \times N}$.}
\end{theorem}

This is proved presently.

\subsection{Stronger Schur monotonicity lemmas}\label{SSstrict}

The proof of Theorem~\ref{Tlmi2} relies on the following strengthening of
the Schur monotonicity lemma above, Theorem~\ref{TSchurmon}.
\editAB{As above, for any set of real numbers} $S$,
denote by $S^N_{\neq}$ the set of all $N$-tuples of distinct elements
of~$S$ and by $S^N_<$ its subset of $N$-tuples with increasing entries.

\begin{theorem}[Schur strict monotonicity lemma 1]\label{Tkt}
Fix an integer $N \geq 1$ and \editAB{distinct} $N$-tuples
$\bm = ( m_0 < \cdots < m_{N - 1} )$ and
$\bn = ( n_0 < \cdots < n_{N-1} )$
in $\R^N_<$ such that $m_j \leqslant n_j$ for all $j$. \editAB{The}
symmetric function 
\[ 
f : ( 0, \infty )^N_{\neq} \to \R; \ \bu \mapsto %
\frac{\det \bu^{\circ \bn}}{\det \bu^{\circ \bm}}
\]
is strictly increasing in each coordinate \editAB{and,
for any $\rho \in ( 0, \infty )$, is}
bounded above by
the constant $\rho^{| \bn - \bm | / 2} V( \bn) / V( \bm )$
on~$( 0, \sqrt{\rho} ]^N_{\neq}$.
\editAB{Furthermore,  if $m_0 = n_0 = 0$ then
$f$ is well defined on $[ 0, \sqrt{\rho} ]^N_{\neq}$
and these two properties hold there.}
\end{theorem}

As announced in Theorem~\ref{T2}, \editAB{an extension
of this result holds for Schur polynomials. This will be stated and
proved below,} after the proof \editAB{of the present theorem}.
We state and prove \editAB{the extended} result separately, because
\editAB{the} behavior of $f$ on the boundary of the orthant is
\editAB{somewhat delicate}.

\begin{proof}
We begin by showing the result on $( 0, \sqrt{\rho} ] ^N_{\neq}$
\editAB{for arbitrary $\rho \in ( 0, \infty )$.}
The first step is to prove that $f$ is strictly increasing in each
coordinate, say in $u_N$. If not, then by Theorem~\ref{TSchurmon},
the function $f$ is constant on
$( u_1, \ldots, u_{N - 1} ) \times [ x, x' ]$ for some
\editAB{$x$, $x' \in ( 0, \sqrt{\rho} ]$ with $x' < x$,}
and we may \editAB{shrink this interval to ensure that $u_j \not\in [ x, x ']$
for $j \neq N$. The function
\[
h : [ \log x, \log x' ] \to \R; \ y %
\mapsto f( u_1, \ldots, u_{N - 1}, e^y ) = %
\frac{\sum_{j = 0}^{N - 1} g_j e^{n_j y}}%
{\sum_{j = 0}^{N - 1} g_j' e^{m_j y}}
\]
is constant, and}
$g_j$ and $g_j'$ are generalized Vandermonde determinants
in $u_1$, \ldots, $u_{N - 1}$ for any $j$, \editAB{so are non-zero.
Since functions of the form $y \mapsto e^{\lambda y}$ are linearly
independent for distinct real $\lambda$, this implies that $\bm = \bn$,
contrary to our initial assumption.

Next, we note that} any vector in $( 0, \sqrt{\rho} ]^N_{\neq}$ is
coordinatewise bounded above
(up to relabeling co\-ordinates) by a vector of the form
\editAB{$\bv = \bv( \epsilon ) := %
\sqrt{\rho} ( 1, \epsilon, \ldots, \epsilon^{N - 1} )^T$,
where $\epsilon \in ( 0, 1 )$.}
Hence, \editAB{by~\ref{TSchurmon}} and~(\ref{Eproduct}),
\[
\frac{\det( \bu^{\circ \bn} )}{\det( \bu^{\circ \bm} )} \leq %
\editAB{\frac{\det( \bv^{\circ \bn} )}{\det( \bv^{\circ \bm} )} = %
\rho^{| \bn - \bm | / 2}} \frac{V( \epsilon^{\bn} )}{V( \epsilon^{\bm} )},
\]
\editAB{where
$\epsilon^{\bn} := ( \epsilon^{n_{i - 1}} )_{i = 1}^N$.
It now} suffices to show \editAB{that}
$V( \epsilon^{\bn} ) / V( \epsilon^{\bm} )$ is bounded above on
$( 0, 1 ]$ by $V( \bn ) / V( \bm )$. As this ratio is non-decreasing in $\epsilon$,
by \editAB{Theorem~\ref{TSchurmon}, the least upper bound
will equal the limit as $\epsilon \to 1^-$, if it exists, but this limit is as claimed,
by L'H\^opital's rule.}
This shows the result on $( 0, \sqrt{\rho} ]^N_{\neq}$.

We now show that $f$ is \editAB{well defined and} strictly increasing at
$\bu \in [ 0, \rho ]^N_{\neq}$, where one coordinate of
$\bu$, say~$u_1$, is zero. Then $m_0 = n_0 = 0$ by assumption,
so the matrices $\bu^{\circ \bm}$ and $\bu^{\circ \bn}$ both have
first row \editAB{$\be_1 := ( 1, 0, \ldots, 0 )$.}
Now if $\bv_1$ denotes the truncation of the vector~$\bv$ by removing its first coordinate, then
\[
f( \bu ) = \frac{\det (\bu_1^{\circ \bn_1})}{\det (\bu_1^{\circ \bm_1})},
\]
by expanding both determinants along their first rows;
\editAB{in particular, $f( \bu )$ is well defined. As}
$\bu_1 \in (0,\sqrt{\rho}]^{N-1}_{\neq}$,
\editAB{the previous working implies that} the
right-hand side is strictly increasing in the coordinates of $\bu_1$,
\editAB{that is,} in all but the first coordinate of $\bu$,
\editAB{and has the requisite upper bound.}

Finally, say $\nu > 0$ and
$\bv := \bu + \nu \be_1 \in ( 0, \infty )^N_{\neq}$;
\editAB{we wish to show that} $f( \bv ) > f( \bu )$.
We may assume that $\nu < \min\{ u_2, \ldots, u_N \}$, by
\editAB{transitivity and the previous working.  Hence}
$\bv( t ) := \bu + t \be_1$ \editAB{is well defined for any
$t \in [ 0, \bu ]$ and we see that}
\[
f( \bv ) = f\bigl( \bv( \nu ) \bigr) >  f\bigl( \bv( \nu / 2 ) \bigr) > %
f\bigl( \bv( t ) \bigr) \qquad \textrm{for any } t \in ( 0, \nu ).
\]
Taking the limit as $t \to 0^+$, it follows that
\[
f(\bv) > \lim_{t \to 0^+} f\bigl( \bv( t ) \bigr) = %
f\bigl( \bv( 0 ) \bigr) = f( \bu ),
\]
\editAB{as desired.}
\end{proof}

The next result is the analogue of Theorem~\ref{Tkt} for ratios of Schur
polynomials on the positive orthant. Given
\editAB{Theorem~\ref{TSchurmon} and} the preceding Theorem~\ref{Tkt},
it is natural to ask if strict monotonicity extends to the boundary of
the orthant $[ 0, \infty )^N$. The following remark \editAB{explains why
this cannot happen and why Theorem~\ref{Tkt2} is the best possible
result that may be obtained.}

\begin{remark}\label{Rfails}
Here we describe two ways in which the coordinatewise monotonicity of the
Schur-polynomial ratio $s_\bn / s_\bm$ fails to extend to strict
monotonicity on all of $[ 0, \infty )^N \setminus \{ \bzero \}$. 

\editAB{Suppose $\bm \in ( \Z_+ )^N_<$ is such that $\bm - \bdelta$}
has exactly $N - k$ non-zero entries, \editAB{where $0 \leq k \leq N$.
Then}
$s_\bm( \bu )$ vanishes \editAB{whenever}
$u_1 = \cdots = u_{k + 1} = 0$, \editAB{so for} every
vector $\bu \in [ 0, \infty )^N$ with at least $k + 1$ coordinates equal to zero.
This is because every semi-standard Young tableau of shape $\bm - \bdelta$
necessarily contains at least one entry \editAB{in the set
$\{ 1, \ldots, k + 1 \}$.} Thus, the ratio $s_\bn( \bu ) / s_\bm( \bu )$
\editAB{has} domain of definition
\begin{equation}\label{EUk}
\cU_k := \{ \bu \in [ 0, \infty )^N : %
\textrm{at most $k$ coordinates of $\bu$ are } 0 \},
\end{equation}
\editAB{as some of monomials in the Littlewood identity (\ref{Elittlewood})
must be non-zero when $\bu \in \cU_k$.

Even restricted} to the domain $\cU_k$, the function
$\bu \mapsto s_\bn( \bu ) / s_\bm( \bu )$ need not be
strictly increasing \editAB{in each coordinate. If}
$\bn - \bdelta$ has \editAB{exactly $l$ zero entries
and} $\bm - \bdelta$ \editAB{has exactly $k$ zero entries,
with $l < k$, then $s_\bn( \bu )$ vanishes whenever
$l + 1$ or more coordinates of $\bu$ are zero,
so $f( \bu ) = s_\bn( \bu ) / s_\bm( \bu )$
vanishes whenever $\bu$ has between $l + 1$
and $k$ coordinates equal to $0$.}
In particular, \editAB{the function $f$ cannot be
strictly increasing on the collection of all such vectors}.
\end{remark}

Given \editAB{the understanding of obstructions to strict monotonicity afforded by} Remark~\ref{Rfails}, we now state and prove the
\editAB{strongest-possible monotonicity result} for ratios of Schur polynomials
\editAB{on the closed orthant $[ 0, \infty )^N$.}

\begin{theorem}[Schur strict monotonicity lemma 2]\label{Tkt2}
Fix an integer $N \geq 1$ and \editAB{distinct $N$-tuples}
$\bm = ( m_0 < \cdots < m_{N - 1} )$ and
$\bn = ( n_0 < \cdots < n_{N - 1} )$ in $( \Z_+ )^N_<$ such that
$m_j \leqslant n_j$ for all $j$.
\begin{enumerate}
\item The symmetric function 
\[
f : ( 0, \infty )^N \to \R; \ \bu \mapsto %
\frac{s_\bn( \bu )}{s_\bm( \bu )}
\]
is strictly increasing in each \editAB{coordinate and,
for any $\rho \in ( 0, \infty )$, is bounded above by
the constant
$\rho^{| \bn - \bm | / 2} V( \bn) / V( \bm ) = f( \sqrt{\rho} \bone_N^T )$
on $( 0, \sqrt{\rho} ]^N$.}

\item Suppose \editAB{that} $n_j = j$ for $j = 0$, \ldots, $k - 1$ but
$n_k > k$, \editAB{where $0 \leq k \leq N$
and the final condition holds vacuously if $k = N$.}
Then $f$ is non-decreasing in each coordinate on its
\editAB{extended} domain of definition
$\cU_k$ \editAB{given by}~(\ref{EUk}).

\item Suppose that $m_j = m_j = j$ for $j = 0$, \ldots, $k - 1$ and
$m_k$, $n_k > k$, \editAB{where $0 \leq k \leq N$
and the final condition holds vacuously if $k = N$.}
Then $f$ is strictly increasing in each
coordinate \editAB{on} $\cU_k$.
\end{enumerate}
\end{theorem}

\begin{proof}
While this result is similar to Theorem~\ref{Tkt}, its proof is slightly different:
the first part uses Schur polynomials rather than \editAB{exponentials},
while the other parts use semi-standard Young tableaux.

\editAB{Given $\bu \in ( 0, \infty )^N_{\neq}$,}
we can expand both determinants
\editAB{along the $N$th row to see that}
\[
f( \bu ) = \frac{\det( \bu^\bn )}{\det( \bu^\bm )} = %
\frac{\sum_{j = 0}^{N - 1} %
( -1 )^{N + j + 1} \det( \bu_0^{\bn^{(j)}} ) u_N^{n_j}}%
{\sum_{j = 0}^{N - 1} %
( - 1)^{N + j + 1} \det( \bu_0^{\bm^{(j)}} ) u_N^{m_j}},
\]
\editAB{where $\bu_0 := ( u_1, \ldots, u_{N - 1} )$,
$\bm^{(j)}$ equals $\bm$ with $m_j$ removed, and
similarly for $\bn^{(j)}$. Dividing numerator and denominator
by the Vandermonde determinant $V( \bu_0 )$, we see that}
\begin{equation}\label{Eschurdet}
f( \bu ) = \frac{\sum_{j = 0}^{N - 1} %
( -1 )^{N + j + 1} s_{\bn^{(j)}}( \bu_0 ) u_N^{n_j}}%
{\sum_{j = 0}^{N - 1} %
( -1 )^{N + j + 1} s_{\bm^{(j)}}( \bu_0 ) u_N^{m_j}}.
\end{equation}
\editAB{As both sides are continuous
on $( 0, \infty )^N$, the identity (\ref{Eschurdet}) holds on the
entire open orthant.}

With~(\ref{Eschurdet}) at hand, we turn to the proof of the theorem.
\begin{enumerate}
\item By symmetry, it suffices to show $f( \bu )$ is strictly increasing
\editAB{as a function of $u_N$.} If not, by Theorem~\ref{TSchurmon}
there exists a point $\bu \in ( 0, \infty )^N$ and some $\epsilon > 0$ such that
\editAB{the function $x \mapsto f( \bu + x \be_N )$ is constant,
say with value $c$, on $[ 0, \epsilon ]$, where
$\be_N := ( 0, \ldots, 0, 1 )$. It}
follows via~(\ref{Eschurdet}) that the \editAB{function}
\[
g : x \mapsto \sum_{j = 0}^{N - 1} ( -1 )^{N + j + 1} %
( s_{\bn^{(j)}}( \bu_0 ) x^{n_j} - %
c s_{\bm^{(j)}}( \bu_0 ) x^{m_j} )
\]
is identically zero \editAB{on $[ 0, \epsilon ]$.
As} $g$ is a non-constant polynomial, since $\bm \neq \bn$,
\editAB{this} yields a contradiction.

\item Let $\bu = ( u_1, \ldots, u_N )^T \in \cU_k$
\editAB{and suppose without loss of generality that
$u_j = 0$ if~$j \leq l$ and $u_j > 0$ if $j > l$,
where $0 \leq l \leq k$. 
Given any $i \in \{ 1, \ldots, N \}$ and $t > 0$, we wish to show that
$f( \bu + t \be_i ) \geq f( \bu )$.
If $\epsilon$ is positive and sufficiently small, we have that
\[
\bu_\epsilon := %
( \epsilon, \ldots, \epsilon, u_{l + 1}, \ldots, u_N )^T = %
\epsilon \sum_{j = 1}^l \be_j + \bu \in ( 0, \infty )^N.
\]}
By Theorem~\ref{TSchurmon}, we know that
\editAB{\[
f( \bu_\epsilon + t \be_i ) - f( \bu_\epsilon ) \geq 0.
\]
We have that $s_\bm( \bu_\epsilon + t \be_i ) > 0$ and
$s_\bm( \bu_\epsilon ) > 0$, and the same holds for
$s_\bm( \bu + t \be_i )$ and $s_\bm( \bu )$, so we may
take $\epsilon \to 0^+$ to obtain the desired}
inequality.

\item \editAB{Let $\bu$, $k$, $l$, $i$ and $t$ be as for (2). We wish to
show that $f( \bu + t \be_i ) > f( \bu )$.

We first suppose $i > l$ and so we may take} $i = l+1$ by symmetry.
We now use the Littlewood \editAB{identity (\ref{Elittlewood}).
As} $u_1 = \cdots = u_l = 0$,
\editAB{the Schur polynomial}~$s_\bn( \bu )$
\editAB{is obtained} by adding monomials
corresponding to all semistandard Young tableau of shape
$\bn - \bdelta$ \editAB{that do not contain any of the
labels} $1$, \ldots, $l$. \editAB{Hence} this sum can be written as a
Schur polynomial in the reduced set of variables
$\bu' := ( u_{l + 1}, \ldots, u_N )^T$
\editAB{and the Littlewood sum involves
tableau of the shape $\bn' - \bdelta$, where
\[
\bn' := ( n_l - l, \ldots, n_{N - 1} - l ).
\]}
In other words,
\[
f( \bu ) = \frac{s_\bn( \bu )}{s_\bm( \bu )} = %
\frac{s_{\bn'}( \bu' )}{s_{\bm'}( \bu' )}
\]
\editAB{and} this last ratio is strictly increasing in each of the variables in
$\bu'$, by \editAB{part~(1). Hence $f( \bu + t \be_i ) > f( \bu )$
for any $i > l$.}

The remaining case is when $1 \leq i \leq l$, so by symmetry we
\editAB{may assume $i = l$. We proceed similarly to} the previous case,
now summing over all semistandard Young tableaux
which do not contain the labels $1$, \ldots, $l - 1$, and form the
Schur polynomials $s_{\bn''}( u_l, \bu' )$ and
$s_{\bm''}( u_l, \bu' )$, where $\bu' \in ( 0, \infty )^{N - l}$ is as in
\editAB{the previous paragraph,}
\begin{align*}
\bm'' & := ( m_{l - 1} - l + 1, \ldots, m_{N - 1} - l  + 1) \\
\textrm{and} \quad \bn'' & := ( n_{l - 1} - l + 1, \ldots, n_{N - 1} - l + 1 ).
\end{align*}
As above, we have \editAB{that}
\[
f( \bu ) = \frac{s_\bn( \bu )}{s_\bm( \bu )} =%
\frac{s_{\bm''}( u_l, \bu' )}{s_{\bn''}( u_l, \bu' )}.
\]
\editAB{Hence if the function $x \mapsto f( \bu + x \be_l )$
is not strictly monotone on $[ 0, t ]$}
then \editAB{there exist $a$, $b \in ( 0, t )$ with $a < b$ such that}
the function
\[
g : [ a, b ] \to \R; \ x \mapsto %
f( 0, \ldots, 0, x, u_{l + 1}, \ldots, u_N ) = %
\frac{s_{\bm''}( x, \bu' )}{s_{\bn''}( x, \bu' )}
\]
\editAB{is constant. However} this contradicts \editAB{part (1).}\qedhere
\end{enumerate}
\end{proof}

\editAB{The following result is used to show that (b) implies (a)
in Theorem~\ref{Tlmi2}(1). The need to ensure the vector $\bu$ has
non-negative entries means that the elementary argument used in
the proof of Theorem~\ref{Tlmi}(1) does not translate to this setting.

\begin{theorem}\label{Tdom}
Let $A \in \cP_N\bigl( [ 0, \infty ) \bigr)$, where $N \geq 1$,
and suppose the rows of $A$ are distinct.
There exists a vector $\bu \in [ 0, \infty )^N$ with distinct entries
such that $A \geq \bu \bu^T$ and~$\bu$
has a zero entry if and only if $A$ has a zero row.
\end{theorem}

The condition that $A$ must have distinct rows in Theorem~\ref{Tdom}
and for corresponding implication in Theorem~\ref{Tlmi}(1)
is necessary as well as sufficient, as the rank-one case shows.
If $A = \bv \bv^*$ for some $\bv \in \C^N$, and
$\bu \in \C^N$ is such that $A \geq \bu \bu^*$, then
$\bu$ is a scalar multiple of $\bv$, by Proposition~\ref{Ppos-constant}.
If $A$ has two equal rows, then two coordinates of $\bv$ are equal,
whence the same holds for $\bu$, and so the
conclusion of Theorem~\ref{Tdom} and the implication in
Theorem~\ref{Tlmi}(1) do not hold.

We note that the full-rank case of Theorem~\ref{Tdom} is immediate,
either by Proposition~\ref{Ppos-constant} or simply because
$A \geq \lambda_1 \Id_N \geq \lambda_1 \bu \bu^T$ for any unit vector $\bu$,
where $\lambda_1$ is the smallest eigenvalue of $A$ and $\Id_N$ is the
$N \times N$ identity matrix. Similarly, the rank-one case is immedate.

\begin{lemma}\label{Ldom}
Let $A$ be a real symmetric $N \times N$ matrix,
where $N \geq 1$, and suppose the vectors
$\bu_1$, \ldots, $\bu_m \in \R^N$
are such that $A \geq \bu_j \bu_j^T$  for all $j$.
If $\bu = \sum_{j = 1}^m \lambda_j \bu_j$ is an arbitrary
convex combination of $\bu_1$, \ldots, $\bu_m$, so that
$\lambda_j \in [ 0, 1 ]$ for all $j$ and
$\sum_{j = 1}^m \lambda_j = 1$, then $A \geq \bu \bu^T$.
\end{lemma}

\begin{proof}
We recall the following elementary Schur-complement property:
for any $\bv \in \R^N$ we have the equivalence
\[
A \geq \bv \bv^T \quad \iff \quad
\begin{pmatrix} A & \bv \\ \bv^T & 1 \end{pmatrix} \geq 0.
\]
Replacing $\bv$ by $\bu_j$ in the right-hand side,
multiplying through by $\lambda_j$
and summing over $j$ gives the result.
\end{proof}

\begin{proof}[Proof of Theorem~\ref{Tdom}]
We have the spectral decomposition
$A = \sum_{j=1}^m \bu_j \bu_j^T$, where
the eigenvectors $\bu_1$, \ldots, $\bu_m \in \R^N$ are orthogonal and
non-zero. We have that $A \geq \bu_j \bu_j^T$ for all $j$
and, by Lemma~\ref{Ldom}, it suffices to show that
some convex combination of
these vectors has non-negative and distinct entries, with a zero
appearing if and only if $A$ has a zero row.

We first note that, for any pair of distinct indices $j$ and $k$ in
$\{ 1, \ldots, N \}$, there exists some eigenvector $\bu_i$ 
whose $j$th and $k$th coordinates are distinct. If this does not
hold for some such pair then the
$\{ j, k \} \times \{ j, k \}$ principal submatrix of $A$ has the form
$\begin{pmatrix} \alpha & \alpha \\ \alpha & \alpha \end{pmatrix}$
for some $\alpha > 0$.
However, if $l \not\in \{ j, k \}$ then, up to a simultaneous re-indexing
of rows and columns, the $\{ j, k, l \} \times \{ j, k, l \}$ minor
of $A$ is such that
\[
0 \leq \det\begin{pmatrix} \alpha & \alpha & a_{j l} \\
 \alpha & \alpha & a_{k l} \\
 a_{j l} & a_{k l} & a_{l l}
\end{pmatrix} = - \alpha ( a_{j l} - a_{k l})^2 \leq 0.
\]
From this it follows that $a_{j l} = a_{k l}$ for all $l \not\in \{ j, k \}$,
which shows that the $j$th and $k$th rows of $A$ are equal.
This contradiction establishes our first observation.

We next consider the affine map
\[
\Psi : \R^{m - 1} \to \R^N; \ %
\bc = ( c_2, \ldots, c_m )  \mapsto \bu_1 + \sum_{j=2}^m c_j \bu_j.
\]
and note that $\Psi( \bc )$ has distinct coordinates if and only if
$p( \bc ) \neq 0$, where
\[
p( \bc ) := %
\prod_{1 \leq k < l \leq N} \bigl( \Psi( \bc )_l - \Psi( \bc )_k \bigr) = %
\prod_{1 \leq k < l \leq N} \Bigl( ( \bu_1)_l - ( \bu_1 )_k + %
\sum_{j = 2}^m c_j \bigl( ( \bu_j )_l - ( \bu_j )_k \bigr) \Bigr).
\]
Thus, $p$ is a polynomial in $c_2$, \ldots, $c_m$ that is a product of 
non-zero factors that are either linear or constant, by the first observation.
It follows that $\Psi( \bc )$ has distinct coordinates for all $\bc$ not in
$p^{-1}( \{ 0 \} )$, which has zero Lebesgue measure.

We now assume that $A$ is irreducible, which implies that $A$
does not have a zero row. By the Perron--Frobenius theorem,
we may take $\bu_1$ to be the
Perron eigenvector, which lies in $( 0, \infty )^N$. We can then
choose a positive but sufficiently small $\epsilon$ so that
$\Psi( \bc )$ has all coordinates positive whenever
$\bc \in ( 0, \epsilon )^{m - 1}$. Since this set has positive
Lebesgue measure, there exists some
$\bc \in ( 0, \epsilon )^{m - 1} \setminus p^{-1}( \{ 0 \} )$
and $\bu' = \Psi( \bc )$ has positive and distinct coordinates.
Finally, we let $\bu = \beta \bu'$, where
$\beta = 1 / ( 1 + c_2 + \cdots + c_m )$.

We next suppose that $A = A_1 \oplus A_2$, where
$A_1$ and $A_2$ have vectors
$\bu_1$ and $\bu_2$ with positive entries such that $A_1 \geq \bu_1 \bu_1^T$
and $A_2 \geq \bu_2 \bu_2^T$. A short calculation shows that
\begin{multline*}
\begin{array}{@{}c@{}}
\begin{bmatrix} \bx_1^T & \bx_2^T \end{bmatrix} \\
\mathstrut
\end{array}
\biggl( \begin{bmatrix} A_1 & 0 \\ 0 & A_2 \end{bmatrix} - %
\begin{bmatrix} \mu_1 \bu_1 \\ \mu_2 \bu_2 \end{bmatrix}
\begin{array}{@{}c@{}}
\begin{bmatrix} \mu_1 \bu_1^T & \mu_2 \bu_2^T \end{bmatrix} \\
\mathstrut
\end{array}
\biggr) \begin{bmatrix} \bx_1 \\ \bx_2 \end{bmatrix} \\[1ex]
 = \bx_1^T A_1 \bx_1 - 2 \mu_1^2 ( \bx_1^T \bu_1 )^2 + %
\bx_2^T A_2 \bx_2 - 2 \mu_2^2 ( \bx_2^T \bu_2 )^2
+ ( \mu_1 \bx_1^T \bu_1 - \mu_2 \bx_2^T \bu_2 )^2,
\end{multline*}
so any $\bu$ of the form $\mu_1 \bu_1 \oplus \mu_2 \bu_2$, with
$\mu_1^2 < 1 / 2$ and $\mu_2^2 < 1 / 2$, is such that $A \geq \bu \bu^T$.
To ensure that $\bu$ has distinct and positive entries, we fix suitable positive
$\mu_1$ and take~$\mu_2$ positive but
sufficiently small to ensure that every entry of $\mu_2 \bu_2$ is smaller
than every every of $\mu_1 \bu_1$. Since $A$ may be written,
up to a simultaneous re-indexing of rows and columns, 
in Frobenius normal form as a block-diagonal sum of irreducible matrices
and at most one zero, the result follows.
If $A$ has a zero row then zero appears in the appropriate
coordinate of $\bu$, and otherwise all the entries of $\bu$ are positive.
\end{proof}}

We \editAB{now use} Theorem~\ref{Tkt} to show \editAB{that}
strict positive definiteness \editAB{holds generically for}
polynomial positivity preservers.

\begin{proof}[Proof of Theorem~\ref{Tlmi2}]
The proof of part~(1) is similar to that of the corresponding
\editAB{parts of the proof}
of Theorem~\ref{Tlmi}, with \editAB{a few minor} modifications.
\editAB{To see that (a) implies (c) here, we note first that if
$A$ has rank one then we may assume $A = \bu \bu^*$.
As before, the matrix $h[ \bu \bu^* ]$ is the sum of $N$ rank-one
matrices and their column spaces are spanned by
$\{ \bu^{\circ n_0}, \ldots, \bu^{\circ n_{N - 1}} \}$.
This set is linearly independent, as the
generalized Vandermonde determinant of
$\bu^{\circ \bn} = ( u_i^{n_{j - 1}} )_{i, j = 1}^N$
is non-zero if $u_1$, \ldots, $u_N > 0$ and
$n_0 < n_1 < \cdots < n_{N - 1}$
\cite[Example~XIII.8.1]{Gant2}. In the case where $u_i = 0$ for some $i$
then the $i$th row of the matrix~$\bu^{\circ \bn}$ equals $( 1, 0, \ldots, 0 )$
and expanding the determinant along this row reduces the matter
to the former situation.
When $A = \bu \bu^*$ we are now done; otherwise
we are in the setting of Definition~\ref{Dsetup}(3)
and we emply Loewner monotonicity as in the proof of Theorem~\ref{Tlmi}.

The arguments to show that (c) implies (b) and that (c) and (d) are equivalent
are unchanged and the fact that (b) implies (a) follows immediately from
Theorem~\ref{Tdom}.}

For part~(2), we first suppose as in the proof of Theorem~\ref{Tlmi}(2)
that $A = \bu \bu^*$ has rank one, and $\det g[\bu \bu^*] = 0$. By suitably
specializing~(\ref{Ejacobi-trudi2}), \editAB{we see that}
\[
\sum_{j = 0}^{N - 1} \frac{\det (\bu^{\circ \bn_j})^2}{c_j} = %
\cC \det (\bu^{\circ \bn})^2 = \det (\bu^{\circ \bn})^2
\sum_{j = 0}^{N - 1} %
\frac{V(\bn_j)^2}{V(\bn)^2} \frac{\rho^{M - n_j}}{c_j},
\]
where $\bn_j$ and $\bn$ are as in~(\ref{Erankone}). Moreover, by the
hypotheses we have $\det (\bu^{\circ \bn}) \neq 0$. Thus,
\begin{equation}\label{Eratio}
\sum_{j = 0}^{N - 1} %
\frac{\det (\bu^{\circ \bn_j})^2}{c_j \det (\bu^{\circ \bn})^2} = %
\sum_{j = 0}^{N - 1} %
\frac{\rho^{| \bn_j - \bn |} V( \bn_j )^2}{c_j V( \bn )^2}.
\end{equation}
By Theorem~\ref{Tkt}, each summand on the left is strictly less than
the corresponding one on the right \editAB{whenever}
$\bu \in [ 0, \rho ]^N_{\neq}$ \editAB{and so $g[ \bu \bu^* ]$
is positive definite.
The remaining case occurs when $\bu$ has a zero entry, in which case
$n_0 = 0$ and
$\bn_0 := ( n_1 < \cdots < n_{N - 1} < M )$ lies in $( 0, \infty )^N_<$.
Then $\bu^{\circ \bn_0}$ has a zero row and therefore zero determinant,
whereas if $\bm$ is such that $m_0 = 0$ then
$\det( \bu^{\circ \bm} )^2 = \det( \bu_\times^{\circ \bm'} )^2$,
where $\bu_\times$ is $\bu$ with the zero entry removed,
so that $\bu_\times \in ( 0, \sqrt{\rho} ]^{N - 1}_{\neq}$,
and $\bm' := ( m_1 < \cdots < m_{N - 1} )$.
Hence
\[
\frac{\det( \bu^{\circ \bn_j} )^2}{\det( \bu^{\circ \bn} )^2} = %
\frac{\det( \bu_\times^{\circ \bn_j'} )^2}%
{\det( \bu_\times^{\circ \bn'} )^2} \leq %
\rho^{| \bn_j' - \bn' |} \frac{V( \bn_j' )^2}{V( \bn' )^2} \leq %
\rho^{| \bn_j - \bn |} \frac{V( \bn_j )^2}{V( \bn)^2}
\]
for $j = 1$, \ldots, $N - 1$. (The final inequality holds because
if $m_k \leq n_k$ for $k = 0$, \ldots, $N - 1$ and $m_0 = n_0$
then $V( \bm ) / V( \bm' ) \leq V( \bn ) / V( \bn' )$.)
Thus the equality (\ref{Eratio}) fails to hold
once again and we see that} $g[\bu \bu^*]$ is positive definite.

\editAB{The proof for general $A$ is identical to that part of}
the proof of Theorem~\ref{Tlmi}(2),
\editAB{and the} same holds for the proof of the final part.
\end{proof}

\editAB{We conclude this section with the following observation.}

\begin{remark}
\editAB{As noted in the introduction, and explained by Loewner}
(see Horn's thesis \cite{horn}), \editAB{a necessary
condition for any smooth function $f : ( 0, \rho ) \to \R$
to preserve positive semidefiniteness}
when applied entrywise
to matrices in $\cP_n\bigl( ( 0, \rho ) \bigr)$
\editAB{is that $f$, $f'$, \ldots, $f^{( n - 1 )}$ must}
be non-negative on $( 0, \rho )$.

Now a natural question is \editAB{as follows}: is Loewner's necessary condition also
sufficient? \editAB{For} power functions \editAB{of the form}
$p(x) \equiv x^\alpha$ then this condition is indeed sufficient, as shown by
FitzGerald and Horn~\cite{FitzHorn}. However, \editAB{this}
necessary condition is not sufficient \editAB{in general}.

\editAB{From Theorem~\ref{Tthreshold} with $c_0$, $c_1 > 0$,
$\bn = ( 0, 1 )$ and $M = 2$,
we see that the quadratic polynomial $p( x ) = c_0 + c_1 x + c' x^2$}
preserves \editAB{positive semidefiniteness} on $\cP_2\bigl( ( 0, 1 ) \bigr)$
if and only if
\begin{equation}\label{Esharp}
c' \geq \frac{-c_0 c_1}{4 c_0 + 2 c_1}.
\end{equation}
\editAB{On} the other hand, Loewner's \editAB{result provides a
lower bound for} the \editAB{coefficient $c'$ which} can be computed as
follows. \editAB{As $p'$ is non-negative on $[ 0, 1 ]$, we have that}
$2 c' x + c_1 \geq 0$ for any $x \in [ 0, 1 ]$, so $c' \geq -c_1 / 2$.
\editAB{If $x \in [ 0, 1]$ then this implies that}
\[
p( x ) = c' x^2 + c_1 x + c_0 \geq \frac{-c_1}{2} x^2 + c_1 x + c_0 = %
c_1 x \Bigl( 1 - \frac{x}{2} \Bigr) + c_0 \geq c_0 > 0.
\]
\editAB{(Alternatively, one may observe that $p$ is non-decreasing
on $[ 0, 1 ]$, since $p'( x ) > 0$ for any choice of $x \in ( 0, 1 )$, 
nd so $f$ is bounded below by $p( 0 ) = c_0$.)}
Thus, \editAB{the} lower bound \editAB{on~$c'$ to ensure that
Loewner's condition holds} is $-c_1/2$, which
\editAB{is} strictly smaller than the bound in~\eqref{Esharp}. Hence
Loewner's necessary condition is not sufficient, even for polynomial functions.
We thank Siddhartha Sahi for raising this question.
\end{remark}

\section{Rank properties on strata}\label{Srank}

Theorem~\ref{Tlmi} \editAB{provides readily verified
criteria to classify when a
matrix}~$A \in \cP_N\bigl( \overline{D}( 0, \rho ) \bigr)$
\editAB{is such that} $f[ A ]$ is non-singular, and also
\editAB{implies that
there are at most two choices of $A$ for which $f[ A ]$ is} zero.
This section significantly refines both of these results, by
\editAB{provding a method to compute the} rank of the
matrix~$f[ A ]$.
\editAB{A tool developed in previous work
\cite{BGKP-pmp,BGKP-strata},}
a Schubert cell-type stratification of the cone $\cP_N( \C )$,
turns out to be \editAB{crucial}:
the rank of $A$ depends solely on which stratum $A$ lies in.
We begin by recalling the relevant notions.

\begin{definition}
Given an integer $N \geq 2$, denote by $( \Pi_N, \preccurlyeq )$ the
poset of all partitions of the set $\{ 1, \ldots, N \}$, ordered
such that $\pi' \preccurlyeq \pi$ if and only if $\pi$ is a refinement of
$\pi'$: \editAB{every set in $\pi$ is a subset of some set in $\pi'$.

We let $| \pi |$ denote the number of sets in $\pi$ and $| I |$ denote
the number of elements in a set $I \in \pi$.
\editAB{We insist that $N \geq 2$ throughout this section
to avoid uninteresting trivialities.}

Given non-empty sets $I$, $J \subseteq \{ 1, \ldots, N \}$ and an
$N \times N$ complex matrix $A$, we let  $A_{I \times J}$ denote the
$| I | \times | J |$ submatrix of $A$ with row indices in $I$
and column indices in $J$.}
\end{definition}

\begin{proposition}[{\cite[Propositions~2.4 and~2.6]{BGKP-strata}}]\label{Pexists}
Fix an integer $N \geq 2$ and a multiplicative
subgroup~$G \leq \C^\times$.
\begin{enumerate}
\item \editAB{For any $N \times N$ complex matrix} $A$,
there exists a unique
\editAB{minimal partition $\pi \in \Pi_N$ such
that the entries of the submatrix $A_{I \times J}$ lie in a
single $G$-orbit for all $I$, $J \in \pi$.}

In particular, there exists \editAB{an $| \pi | \times | \pi |$ complex matrix}
$C$ such that $A$ is a block matrix with
$A_{I \times J} = c_{I J} \bone_{| I | \times | J |}$ for all
$I$, $J \in \pi$. Moreover, $A$ and $C$ have equal rank.

\item \editAB{There is a} stratification of the set of $N \times N$
complex matrices,
\[
\C^{N \times N} = \bigsqcup_{\pi \in \Pi_N} \stratumsymb^G_\pi,
\]
\editAB{where} the \emph{stratum}
\[
\stratumsymb^G_\pi := \{ A \in \C^{N \times N} : \pi^G( A ) = \pi \}
\]
\editAB{and $\pi^G( A )$ is the partition from (1). The}
set $\stratumsymb^G_\pi$ has closure
\begin{equation}\label{Eschubert}
\overline{\stratumsymb^G_\pi} = %
\bigsqcup_{\pi' \preccurlyeq \pi} \stratumsymb^G_{\pi'}
\end{equation}
when $\C^{N \times N}$ is equipped with its usual topology.
\end{enumerate}
\end{proposition}

Using the above isogenic block stratification, we now refine the results
in the preceding section. We \editAB{let}
$\pi_\vee := \{ \{ 1 \}, \ldots, \{ N \} \}$ \editAB{denote}
the maximum element of the lattice of partitions $\Pi_N$
and we work henceforth with $\pi^G( A )$ only for
the trivial subgroup $G = \{ 1 \}$.
\editAB{To lighten notation, we write
$\stratumsymb_\pi^{\{1\}} = \stratumsymb_\pi$
and $\pi^{\{1\}}( A ) = \pi( A )$.}

\begin{theorem}\label{Tstrata-rank}
\editAB{Let $f$ be as in Definition~\ref{Dsetup}(4), so
that $n_0$ and $M$ are non-negative integers and $n_j = n_0 + j$
for $j = 0$, \ldots, $N -1$, where $N \geq 2$.
Suppose $c_0$, \ldots, $c_{N - 1} > 0$
and $c' > -\cC^{-1}$,
where $\cC$ is as in (\ref{DC}).}
\editAB{Let} $A \in \cP_N\bigl( \overline{D}( 0, \rho ) \bigr)$,
with $n_0 = 0$ if $A$ has a zero row. Then
\begin{equation}
\rk f[ A ] = \rk h[ A ] = | \pi( A ) |,
\end{equation}
while $\rk g[ A ] = | \pi( A ) |$ if
\begin{enumerate}
\item[(a)] $A \not \in \stratumsymb_{\pi_\vee}$ or
\item[(b)] $A \in \stratumsymb_{\pi_\vee}$ and $A$ has a row
with distinct entries, with $n_0 = 0$ if any entry in this row is zero.
\end{enumerate}
\end{theorem}
In particular, for any partition $\pi \in \Pi_N$ and any positive
semidefinite matrix $A \in \stratumsymb_\pi$, both $f[ A ]$
and $h[ A ]$ have rank equal to the number of blocks in $\pi$
\editAB{(as long as $n_0 = 0$ whenever $A$ has a zero row).}

When $N > 2$, the identity matrix is an element of
$\stratumsymb_{\pi_\vee}$ which has no row with distinct
entries. It follows that \editAB{Theorem~\ref{Tstrata-rank}(b)}
is a sufficient but not necessary
condition for \editAB{the rank of} $g[ A ]$ to equal~$| \pi( A )|$.

\begin{remark}\label{Rmore-equiv}
Theorem~\ref{Tstrata-rank} is intertwined with Theorem~\ref{Tlmi} in
two ways. First, \editAB{the matrices}
$f[ A ]$ and $h[ A ]$ have rank equal to
$|\pi( A )|$, so are never zero.
Second, the four equivalent
assertions in Theorem~\ref{Tlmi}(1) are also equivalent to
\editAB{the following}:
\begin{enumerate}
\item[(e)] \editAB{The matrix $A$ lies in} $\stratumsymb_{\pi_\vee}$, the top cell of the stratification.
\end{enumerate}
\editAB{Since $\stratumsymb_{\pi_\vee}$ is dense in $\cP_N$, we see
again that $f[ A ]$ is positive definite for generic $A$.}
\end{remark}

The proof of Theorem~\ref{Tstrata-rank} employs the block decomposition
of Proposition~\ref{Pexists}, as well as the inflation and compression
operators for the entrywise calculus studied elsewhere
\cite[Section~4]{BGKP-strata},
\editAB{\cite{BGKP-inertia}}. We begin by recalling these operators and
some basic properties.

\begin{definition}[{\cite[Definition~4.1]{BGKP-strata}}]\label{Dcomp}
\editAB{Suppose $\pi = \{ I_1, \ldots, I_m \} \in \Pi_N$ for some
$N \geq 1$.
Given $i$, $j \in \{ 1, \ldots, m \}$, we} let $E_{i j}$ denote the
elementary \editAB{$m \times m$}
matrix with $( i, j )$ entry equal to~$1$ and all other
entries~$0$, and let \editAB{$\bone[ I_i \times I_j ]$ denote} the
$N \times N$ matrix with $1$ in each entry of the $I_1 \times I_j$
block and $0$ elsewhere.

\begin{enumerate}
\item Define the linear \emph{inflation} map
\[
\up : \C^{m \times m} \to \C^{N \times N}; \ %
E_{i j} \mapsto \bone[ I_i \times I_j ] \quad ( i, j = 1, \ldots, m )
\]
and note that the range of $\up$ is $\overline{\stratum{\pi}}$.

\item Define the linear \emph{compression} map
\[
\down : \C^{N \times N} \to \C^{m \times m}; \ %
\down( A )_{i j} := 
\frac{1}{| I_i | \, | I_j |} \sum_{p \in I_i, q \in I_j} a_{p q} %
\quad \editAB{( i, j = 1, \ldots, m ),}
\]
so that the image \editAB{$\down( A ) = ( b_{i j} )_{i, j = 1}^m$}
is such that $b_{i j}$ is the
average of the entries in~$A_{I_i \times I_j}$.
\end{enumerate}
\end{definition}

The operators $\up$ and $\down$ are well behaved with respect to the entrywise
calculus:

\begin{theorem}[{\cite[Theorem~4.2]{BGKP-strata}}]\label{Tentrywise}
Let $\overline{\stratum{\pi}}$ and $\C^{m \times m}$
\editAB{each} be equipped with
the entrywise product, so that the units for this product are
$\one{N}$ and $\one{m}$, respectively.
The maps
\[
\down : \overline{\stratum{\pi}} \to \C^{m \times m} %
\qquad  \text{and} \qquad %
\up : \C^{m \times m} \to \overline{\stratum{\pi}}
\]
are mutually inverse, rank-preserving isomorphisms of unital commutative
$*$-algebras. Moreover, $A \in \overline{\stratum{\pi}}$ is positive
semidefinite if and only if~$\down( A )$ is.
\end{theorem}

\editAB{To summarize the preceeding material in} plain language,
the main picture adapted to the trivial group $G = \{ 1 \}$ is the following:
a \editAB{real symmetric} matrix $A = ( a_{i j} )_{i, j = 1}^N$ \editAB{respects the
block structure} associated to a partition $\pi = \{ I_1, \ldots, I_m\}$ if the
\editAB{entry $a_{i j}$ is independent of} $i$, $j \in I_k$ for some $k$.
The compression map collapses each cell $I_k$ to a single  entry, projecting the \editAB{matrix}~$A$ to the \editAB{$m \times m$ matrix with entries given by the constant} values along the fibres of the
projection map. The reverse inflation map restores the repetitions
of matrix entries in $A$. \editAB{These are linear, mutually inverse maps that preserve rank,
positive semidefiniteness, and the entrywise product.}

With these tools at hand, we proceed.

\begin{proof}[Proof of Theorem~\ref{Tstrata-rank}]
\editAB{As} $f$ is equal to $h$ when $c' = 0$, we need only
consider $f[ A ]$ and~$g[ A ]$. For convenience, we let $\pi := \pi( A )$.

\editAB{Suppose $A = a \one{N}$ for some $a \geq 0$. Then
$g( a ) > 0$ if $a > 0$,} by the last part of the proof of 
Theorem~\ref{Tlmi}, \editAB{and $g( 0 ) = c_0 > 1$ when
$A = \bzero_{N \times N}$. Since $f( a ) \geq g( a )$,
the rank-one case is established.}

Next, we note that $\pi = \pi_\vee$ if and only if the rows of $A$ are
distinct, so the result follows from Theorem~\ref{Tlmi} in this case.

\editAB{Otherwise, we suppose} that $m := | \pi |$ is strictly between
$1$ and $N$. \editAB{The matrix $g[ A ]$ is positive semidefinite,
by Theorem~\ref{Tthreshold}, and therefore, if
$B := \down( A )$, so is $g[ B ] = \down\bigl( g[ A ] \bigr)$,
by Theorem~\ref{Tentrywise}.
If $g[ B ]$ is positive definite then
so is $f[ B ]$, since $f[ B ] \geq g[ B ]$, and therefore
both of these matrices have rank $m$. Another application of
Theorem~\ref{Tentrywise} then gives that the matrices
$g[ A ] = \up\bigl( g[ B ] \bigr)$
and $f[ A ] = \up\bigl( g[ A ] \bigr)$ have rank $m$, as
required.}

It thus remains to show that $g[ B ]$ is positive definite. For this, we
will use Lemma~\ref{Llmi} with
$C = h[ B ]$ and $D = \cC_m^{-1} B^{\circ M}$ for a suitable 
positive scalar $\cC_m$. The Schur product theorem gives that
\editAB{$C$ and $D$ are both positive semidefinite.
Furthermore, as $B$ has distinct rows and $n_0 = 0$ if
$B$ has a zero row, Theorem~\ref{Tlmi}(1) gives that
\[
C_0 := \sum_{j = 0}^{m - 1} c_j B^{\circ n_j}
\]}
is positive definite.
\editAB{Since $h[ B ] = C \geq C_0$, we have that $h[ B ]$ is positive
definite as well. We now let
\[
h_m( z ) := \sum_{j = 0}^{m - 1} c_{N - m + j} z^j
\]
and let $\cC_m$ equal $\cC$ as in (\ref{DC}) but with $N$ replaced with $m$,
$\bn = ( 0, \ldots, m - 1)$, $M$ replaced by $M - n_0 - N + m$
and $\bc = ( c_{N - m}, \ldots, c_{N - 1} )$, so that
\[
\cC_m = \sum_{j = 0}^{m - 1}
{M - N + m \choose j}^2 {M - N + m - j - 1 \choose m - j - 1}^2 %
\frac{\rho^{M - n_0 - N + m - j}}{c_{N - m + j}}.
\]
By Corollary~\ref{Clmi}},
\[
h_m[ B ] \geq \cC_m^{-1} B^{\circ ( M - n_0 - N + m )}
\]
\editAB{and therefore,} by the Schur product theorem, we have that
\begin{equation}
C = h[ B ] \geq B^{\circ ( n_0 + N - m)} \circ h_m[ B ] \geq %
\cC_m^{-1} B^{\circ M} = D.
\end{equation}
Moreover, $\cC_m < \cC_{m + 1} \leq \cC_N$,
\editAB{where the} constant $\cC_N$
is precisely \editAB{$\cC$ as in (\ref{DC})}, 
\editAB{since
\[
{M - N + m + 1 \choose j + 1} {M - N + m \choose j}^{-1} = %
\frac{M - N + m + 1}{j + 1} > 1 \qquad \textrm{for } j = 0, \ldots, m - 1.
\]
Hence}
\[
g[ B ] = h[ B ] - \cC_N^{-1} B^{\circ M} = C - t D,
\]
where $t = \cC_m / \cC_N \in ( 0, 1 )$. By Lemma~\ref{Llmi}, this has the
same rank as \editAB{$C = h[ B ]$,} which was shown above to be positive definite.
This completes the proof.
\end{proof}

As \editAB{in the previous section, there is an analogue of}
Theorem~\ref{Tstrata-rank} \editAB{that holds in the other cases
set out in Definition~\ref{Dsetup}, in the same way that Theorem~\ref{Tlmi}
becomes Theorem~\ref{Tlmi2}.}

\begin{theorem}
\editAB{Let $f$ and $\cP_0$ be as in Definition~\ref{Dsetup}(1--3)
and suppose $c_0$, \ldots, $c_{N - 1} > 0$ and $c' > -\cC^{-1}$,
where $\cC$ is as in (\ref{DC}). Let $A \in \cP_0$ and suppose $n_0 = 0$
if $A$ has a zero row.}
The conclusions of Theorem~\ref{Tstrata-rank} hold once again.
\end{theorem}

\begin{proof}
The proof \editAB{proceeds in the same manner as}
that of Theorem~\ref{Tstrata-rank}, \editAB{with appeals to
Theorem~\ref{Tlmi}} replaced by employing Theorem~\ref{Tlmi2}
\editAB{in its place. As there, it suffices to assume that $m = | \pi |$
is strictly between $1$ and $N$, and show} that the positive
semidefinite matrix $g[ B ]$ is in fact positive definite,
where $B = \down(A)$. Here are the steps of the proof, modified to work
for Settings (1)--(3) in Definition~\ref{Dsetup}.

\editAB{To see that $C := h[ B ]$ and $D := \cC_m B^{\circ M}$ are
positive semidefinite, where $\cC_m$ is a positive constant to be determined,
we use the result of FitzGerald and Horn \cite{FitzHorn} that
the function~$x \mapsto x^\alpha$ acts entrywise to
preserve positive semidefiniteness}
on $N \times N$ real matrices with positive entries whenever
$\alpha \in \Z_+ \cup [ N - 2, \infty )$. 
\editAB{As above, the matrix $C_0$ is 
positive definite, now} by Theorem~\ref{Tlmi2}(1), and hence so is $C$.

\editAB{We now let}
\[
h_m( z ) := \sum_{j = N - m}^{N - 1} c_j z^{n_j}
\]
\editAB{and take $\cC_m$ to be as in (\ref{DC}) with $N = m$,
$\bn = ( n_{N - m}, \ldots, n_{N - 1} )$, $M$ unchanged and
$\bc = ( c_{N - m}, \ldots, c_{N - 1} )$.}
Once again using the result from \cite{FitzHorn},
\editAB{together with Corolllary~\ref{Clmi},} we have that
\[
C := h_N[ B ] \geq h_m[ B ] \geq \cC_m^{-1} B^{\circ m} =: D.
\]
We now claim that $\cC_m < \cC_N$; given this,
the proof is then completed as for Theorem~\ref{Tstrata-rank}.

To show this claim, we note that
\[
\cC_N = \sum_{j = 0}^{N - 1} b_j^2 \frac{\rho^{M - n_j}}{c_j} %
\qquad \text{and} \qquad %
\cC_m = \sum_{j = N - m}^{N - 1} a_j^2 \frac{\rho^{M - n_j}}{c_j},
\]
\editAB{where
\[
b_j = \prod_{k \in \{ 0, \ldots, N - 1 \} \setminus \{ j \}} %
\Bigl( \frac{M - n_k}{n_j - n_k} \Bigr)^2 \quad \textrm{and} \quad %
a_j = \prod_{k \in \{ N - m, \ldots, N - 1 \} \setminus \{ j \}} %
\Bigl( \frac{M - n_k}{n_j - n_k} \Bigr)^2
\]}
Hence
\[
\cC_N - \cC_m \geq \sum_{j=N-m}^{N-1} (b_j^2 - a_j^2) \frac{\rho^{M -
n_j}}{c_j},
\]
so it suffices to show that $b_j^2 > a_j^2$ for $j \geq N - m$. This
holds because
\[
\frac{b_j^2}{a_j^2} = \prod_{k = 0}^{N - m - 1} %
\Bigl( \frac{M - n_k}{n_j - n_k} \Bigr)^2 > 1
\]
\editAB{since} $M > n_j > n_k$ for $j \geq N - m$.
\end{proof}

\section{Continuity of the Rayleigh quotient on strata}\label{Scont}

As well as its relevance for calculating the rank, as seen in
Section~\ref{Srank}, it was shown in~\cite{BGKP-fixeddim}
that the constant-block stratification of Proposition~\ref{Pexists} plays
a crucial role in studying \editAB{the following} Rayleigh quotient:
\begin{equation}\label{ErayleighIntro}
R = R( A, \bu, \bc, M ) := \frac{\bu^* A^{\circ M} \bu}%
{\bu^* ( c_0 A^{\circ n_0} + c_1 A^{\circ n_1} + \cdots + %
c_{N - 1} A^{\circ n_{N - 1}} ) \bu}.
\end{equation}

This Rayleigh quotient is connected to the \editAB{isogenic}
stratification of the cone $\cP_n(\C)$, and this theme was
\editAB{developed} in \cite[Sections~4 and~5]{BGKP-fixeddim}
(for consecutive non-negative integer
exponents) and later in \cite[Section~11]{KT} (for more general exponents).

The optimisation of~(\ref{ErayleighIntro}) gives an \editAB{alternative}
approach for establishing Theorem~\ref{Tthreshold}. Namely,
\editAB{if the coefficients $c_0$, \ldots, $c_{N - 1}$ are positive
and the exponents $n_0$, \ldots, $n_{N - 1}$ are non-negative then,
given any $A \in \cP_N\bigl( ( 0, \rho ) \bigr)$, or
$A \in \cP_N( \C )$ if the exponents are integral,}
there exists a constant $\cC' \geq 0$ such that
\[
A^{\circ M} \leq \cC' \sum_{j = 0}^{N - 1} c_j A^{\circ n_j} = %
\cC' h[ A ]
\]
The smallest such constant $\cC_R = \cC_R( A, h, M )$ may be regarded
as a Rayleigh quotient, and it was shown in
\editAB{\cite[Remark~4.6]{BGKP-fixeddim}
and \cite[Proposition~11.1]{KT}}
that
\begin{equation}\label{Erayleigh}
\cC_R = %
\varrho( h[ A ]^{\dagger / 2} A^{\circ M} h[ A ]^{\dagger / 2} ),
\end{equation}
where $B ^{\dagger / 2} := ( B^\dagger )^{1 / 2}$
\editAB{for any square matrix $B$, with $B^\dagger$ the
Moore--Penrose pseudo-inverse of $B$,}
and $\varrho( \cdot )$ denotes the spectral radius.

\editAB{If} $A = \bu \bu^T$ for a vector
$\bu \in ( 0, \infty )^N_{\neq}$ then $h[ \bu \bu^T ]$ is invertible,
\editAB{since} the generalized Vandermonde matrix $\bu^{\circ \bn}$ is,
and
\[
\cC_R = %
( \bu^{\circ M} )^T h[ \bu \bu^T ]^{-1} \bu^{\circ M} = %
\sum_{j = 0}^{N - 1} %
\frac{( \det \bu^{\circ \bn_j} )^2}{c_j ( \det \bu^{\circ \bn} )^2};
\]
\editAB{see \cite[Corollary~4.5]{BGKP-fixeddim} and
\cite[Proposition~11.2]{KT}.}
This explains the connection to the
sharp threshold in Theorem~\ref{Tthreshold}.

We recall from \cite{BGKP-fixeddim, KT} that an alternate approach to
proving Theorem~\ref{Tthreshold} is to find the maximum of the
bound~(\ref{Erayleigh}) over all $A$
\editAB{in the relevant test set~$\cP_0$.}
The difficulty with this approach lies in the fact that the Rayleigh-quotient
map is not continuous when \editAB{crossing} strata.

\editAB{Our focus in this section is} on the  bound~(\ref{Erayleigh}) for
a single matrix $A$. We are not concerned with the radius $\rho$
\editAB{that appeared previously and}
we do not insist that $M > n_{N-1}$, only that $M > n_0 = 0$.
\editAB{In this setting we obtain} continuity of the Rayleigh quotient on each individual stratum.

\begin{theorem}\label{Textremecriticalvaluecontinuity}
\editAB{Let $h( z ) = \sum_{j = 0}^{N - 1} c_j z^{n_j}$,
where $N \geq 1$, the coefficients $c_0$, \ldots, $c_{N - 1}$
are positive and the exponents
$n_0$, \ldots, $n_{N - 1} \in \Z_+ \cup [ N - 1, \infty )$
are distinct, with $n_0 = 0$.
Fix~$M > 0$ and let $\cP_0 := \cP_N( \C )$
if the exponents $n_0$, \ldots, $n_{N - 1}$
and $M$ are integers and otherwise let
$\cP_0 := \cP_N\bigl( [ 0, \infty ) \bigr)$.}
The map $A \mapsto \cC_R( A, h, M )$ is continuous
on~$\cP_0 \cap \stratum{\pi}$
for any partition $\pi \in \Pi_N$.
\end{theorem}

The proof employs weighted variants of the inflation and compression
operators used in Section~\ref{Srank}
\editAB{that were introduced in} \cite{BGKP-strata}.

\begin{definition}[\cite{BGKP-strata}]\label{Dweighted}
\editAB{Given a partition $\pi = \{ I_1, \ldots, I_m \} \in \Pi_N$,
where $N \geq 2$, we
use} the diagonal matrix
$D_\pi := \diag\bigl( | I_1 |, \ldots, | I_m | \bigr)$
\editAB{to define} the linear operators
\begin{align*}
\ccdown : \C^{N \times N} \to \C^{m \times m}; \ %
& A \mapsto D_\pi^{1 / 2} \down( A ) D_\pi^{1 / 2} \\[1ex]
\text{and} \quad %
\ccup : \C^{m \times m} \to \C^{N \times N}; \ %
& B \mapsto \up( D_\pi^{-1 / 2} B D_\pi^{-1 / 2} ).
\end{align*}
\end{definition}

Just as $\up$ and $\down$ work well with the entrywise calculus, the maps
$\ccup$ and $\ccdown$ are well behaved with respect to the functional
calculus, \editAB{as the following result demonstrates.}

\begin{theorem}[{\cite[Theorem~5.2]{BGKP-strata}}]\label{Tfunctional}
The maps $\ccdown$ and $\ccup$ are mutually inverse, rank-preserving
isomorphisms between the unital $*$-algebras
$\overline{\stratum{\pi}}$ and $\C^{m \times m}$ equipped with the
usual matrix multiplication. Moreover,\editAB{a matrix}
$A \in \overline{\stratum{\pi}}$
is positive semidefinite if and only if $\ccdown(A)$ is.
\end{theorem}

With these preliminaries at hand, we proceed.

\begin{proof}[Proof of Theorem~\ref{Textremecriticalvaluecontinuity}]
Suppose $A \in \cP_0 \cap \stratum{\pi}$ and let
$H := h[ A ]$ for brevity. \editAB{As $A$ is positive semidefinite, so
is $B = \down( A )$, which has distinct rows by construction.
We have that~$h[ B ]$ has no zero row, since $n_0 = 0$,
so either} Theorem~\ref{Tlmi} or Theorem~\ref{Tlmi2}
\editAB{implies that} $\down( H ) = h[ B ]$ is positive definite,
\editAB{where this identity holds} by Theorem~\ref{Tentrywise}.
\editAB{Hence the matrix $\down( H )$ has full rank, and therefore
so does} $\ccdown( H ) = \ccdown( \up( \down( H ) ) )$,
\editAB{by Theorems~\ref{Tentrywise} and~\ref{Tfunctional}.
The matrix $\ccdown( H )$ is therefore invertible, and}
\[
H^\dagger = \ccup( \ccdown( H )^\dagger ) = \ccup( \ccdown( H )^{-1} )
\]
by Theorem~\ref{Tfunctional}. Hence,
\[
H^{\dagger / 2} A^{\circ M} H^{\dagger / 2} = %
\ccup( \ccdown( H )^{-1 / 2} \ccdown( A^{\circ M} ) %
\ccdown( H )^{-1 / 2} ),
\]
and since all the operations $A \mapsto H = h[ A ]$,
$A \mapsto A^{\circ M}$, $B \mapsto B^{-1 / 2}$, $\ccdown$,
$\ccup$ and $\varrho( \cdot )$ are continuous, this gives the claim.
\end{proof}

We conclude with \editAB{two questions. A version of the first
was originally posed in} \cite{BGKP-fixeddim}.

\begin{question}\label{Rrayleigh}
When is the Rayleigh-quotient inequality an equality? More precisely,
\editAB{given $h( z ) = \sum_{j = 0}^{N - 1} c_j z^{n_j}$,
where $N \geq 1$, the coefficients $c_0$, \ldots, $c_{N - 1}$ are
positive and the exponents $n_0 < \cdots < n_{N - 1} < M$ lie in
$\Z_+ \cup [ N - 1 ,\infty )$, when is
$A \in \cP_N\bigl( [ 0, \rho ] \bigr)$ such that the inequality}
\[
\cC_R = %
\varrho( h[ A ]^{\dagger / 2} A^{\circ M} h[ A ]^{\dagger / 2} ) %
\leq \cC_V
\]
\editAB{is an equality, where $\cC_V = \cC$ as in (\ref{DC})? We see from
Theorems~\ref{Tlmi}(2) and \ref{Tlmi2}(2) that equality is not attained if
$A$ has a row with distinct entries, so lies in in the top
stratum~$\stratum{\pi_\vee}$ (and $n_0 = 0$ if
any entry in this row is zero), since
this implies that the matrix $g[ A ] = h[ A ] - \cC_V^{-1} A^{\circ M}$
is positive definite and $h[ A ] - \cC_R^{-1} A^{\circ M}$
is not, because
\[
\bu^* h[ A ]^{1 / 2} %
( \Id_N - \cC_R^{-1} h[ A ]^{-1 / 2} A^{\circ M} h[ A ]^{-1 / 2} ) %
h[ A ]^{1 / 2} \bu = 0
\]
if $\bu = h[ A ]^{-1 / 2} \bv$ and $\bv$ is an eigenvector
corresponding to the maximum eigenvalue
of~$h[ A ]^{-1 / 2} A^{\circ M} h[ A ]^{-1 / 2}$.}
\end{question}

For our next question, we first present another extension of
Theorem~\ref{Tlmi}. \editAB{This} result and its proof involve the linear
matrix inequality~(\ref{Elmi}), in which the matrix $A^{\circ M}$ is
bounded above by powers of lower order. When restricted to the
closure of a particular stratum, this inequality can be strengthened to
involve fewer terms.

\begin{proposition}\label{Pblocklmi}
\editAB{Let the partition $\pi \in \Pi_N$, where $N \geq 2$,
and suppose $\pi$ has $m$~blocks, where $m \geq 1$.
Suppose $c_0$, \ldots, $c_{m - 1}$ are positive and
$n_0$, \ldots, $n_{m - 1}$, $M \in \Z_+ \cup [ N - 1, \infty )$
are distinct, with $n_0 < \cdots < n_{m - 1} < M$.
Given any $\rho > 0$, we
let $\cP_0$ equal~$\cP_N\bigl( \overline{D}( 0, \rho ) \bigr)$
if $n_0$, \ldots, $n_{N - 1}$ and $M$ are integers
and $\cP_N\bigl( [ 0, \rho ] \bigr)$ otherwise.}
We have the bound
\begin{equation}\label{Elmi2}
A^{\circ M} \leq \cC_m \sum_{j =0 }^{m - 1} c_j A^{\circ n_j} %
\qquad \textrm{for all } A \in \cP_0 \cap \overline{\stratum{\pi}},
\end{equation}
\editAB{where $\cC_m$ equals $\cC$ as in (\ref{DC}) with
$\bc = ( c_0, \ldots, c_{m - 1} )$ and
$\bn = ( n_0, \ldots, n_{m - 1} )$.}
Equality is achieved if and only if either $m = 1$ and
$A = \rho \one{N}$, or $n_0 > 0$ and $A = \bzero_{N \times N}$.

Furthermore, if $\cC_m $ is replaced by any larger
constant, and $n_0 = 0$ if $A \in \cP_0 \cap \stratum{\pi}$
has a zero row, then the inequality~(\ref{Elmi2}) is strict for $A$
upon applying $\down$.
\end{proposition}

\begin{proof}
By Theorem~\ref{Tentrywise}, the maps $\up$ and $\down$ can be used to
transfer the setting to~\editAB{either}
$\cP_m\bigl( \overline{D}( 0, \rho ) \bigr)$ \editAB{or
$\cP_m\bigl( [ 0, \rho ] \bigr)$.}
The assertions then follow directly from their counterparts in
Corollary~\ref{Clmi} and Theorems~\ref{Tlmi} and~\ref{Tlmi2};
\editAB{the final statement holds by (1)(d) of each.}
\end{proof}

\begin{question}
An explicit expression for the supremum
of the function $A \mapsto \cC_R( A, h, M )$
on each stratum~$\stratum{\pi} \cap \cP_0$
is known for $\pi = \pi_\wedge$
\cite[Corollary~4.5]{BGKP-fixeddim}
and $\pi = \pi_\vee$
\cite{BGKP-fixeddim, KT}
\editAB{since~$\stratum{\pi_\vee}$ contains
all matrices of the form $A = \bu \bu^T$ where $\bu \in ( 0, \infty )^N$
has distinct coordinates, and so the supremum of $\cC_R( A, h, M )$
is at least, so exactly, $\cC$ from Theorem~\ref{Tthreshold}.}
A natural conjecture, supported by Proposition~\ref{Pblocklmi}, is that the
supremum depends only on the number of blocks in the partition $\pi$
\editAB{and not on any further data from $\pi$.}
\end{question}

\subsection{Acknowledgements}
\editAB{A.B.~was partially supported by Lancaster University while
the intial phase of this work was carried out.}

D.G.~was partially supported by a University of Delaware Research
Foundation grant, by a Simons Foundation
collaboration grant for mathematicians, and by a University of
Delaware Research Foundation Strategic Initiative grant.

A.K.~was partially supported by the Ramanujan Fellowship SB/S2/RJN-121/2017,
MATRICS grant MTR/2017/000295, and SwarnaJayanti Fellowship grants
SB/SJF/2019-20/14 and DST/SJF/MS/2019/3 from SERB and DST (Govt.~of
India), grant F.510/25/CAS-II/2018(SAP-I) from UGC (Govt.~of India),
a Young Investigator Award from the Infosys Foundation,
a Shanti Swarup Bhatnagar Award from CSIR (Govt.\ of India), and
the DST FIST program 2021 [TPN--700661].

M.P~was supported by a  Simons Foundation collaboration grant for mathematicians.

\editAB{The authors thank the Institute for Advanced Study, Princeton
\editAB{and the American Institute of Mathematics, Pasadena}
for \editAB{their} hospitality while this work was concluded.}

\subsection{List of symbols}

We collect below some notation used throughout the text.

\begin{itemize}
\item $\overline{D}( 0, \rho )$ is the closed disc in $\C$ with radius
$\rho$ centered at the origin.

\item $\cP_N^k( I )$ is the set of positive semidefinite $N \times N$
matrices of rank at most $k$ with entries in the set $I \subset \C$.
Such matrices are necessarily Hermitian.

\item $\cP_N( I ) := \cP_N^N( I )$.

\item $\bone_{N \times N'}$ is the $N \times N'$ matrix with each entry
equal to~$1$.

\item $f[ A ]$ is the matrix obtained by applying the
function~$f$ to each of the entries of the matrix~$A$.

\item $A^{\circ \alpha}$ is the matrix obtained by taking the
$\alpha$th power of each of the entries of the matrix $A$,
whenever this is well defined.

\item $\bu^\alpha = ( u_i^\alpha )_{i = 1}^m$
for any real number $\alpha$ and column vector 
$\bu = ( u_i )_{i = 1}^m$ whenever the entries are well defined.

\item $\bu^{\circ \bn} = ( u_i^{n_j} )_{i, j = 1}^m$
for any column vector $\bu = ( u_i )_{i = 1}^m$ and
row vector $\bn = ( n_1, \ldots, n_m )$
whenever these quantities are well defined.

\item $V( \bu )$ is the Vandermonde determinant
of the column vector $\bu = ( u_i )_{i = 1}^m$
or the row vector $\bu = ( u_1, \ldots u_m )$, so that
$V( \bu) = \prod_{1 \leq k < l \leq m} ( u_l - u_k )$.

\item $S^N_{\neq}$ is the collection of all $N$-tuples in $S$
with distinct entries and $S^N_<$ the subset of $S^N_{\neq}$ consisting
of $N$-tuples with strictly increasing entries.

\item $A^\dagger$ is the Moore--Penrose
pseudo-inverse of the matrix~$A$.

\item $\varrho( A )$ is the spectral radius of the matrix $A$.

\item $( \Pi_N, \preccurlyeq )$ is the poset of partitions of
$\{ 1, \ldots, N \}$, where $\pi' \preccurlyeq \pi$ if~$\pi$ is a
refinement of~$\pi'$, so that every set in $\pi$ is a subset of
some set in $\pi'$.

\item $D_\pi$ is the $m \times m$ diagonal matrix with $( i , i)$
entry $| I_i| $, where $\pi = \{ I_1, \ldots, I_m \} \in \Pi_N$.

\item $\down$ and $\up$ are defined in Definition~\ref{Dcomp}.

\item $\ccdown$ and $\ccup$ are defined in Definition~\ref{Dweighted}.

\end{itemize}



\end{document}